\documentclass[sn-mathphys,Numbered]{sn-jnl}

\usepackage{color}
\usepackage{hyperref}
\usepackage{epsfig}
\usepackage{graphicx}%
\usepackage{multirow}%
\usepackage{amsmath,amssymb,amsfonts}%
\usepackage{amsthm}%
\usepackage{mathrsfs}%
\usepackage[title]{appendix}%
\usepackage{xcolor}%
\usepackage{textcomp}%
\usepackage{manyfoot}%
\usepackage{booktabs}%
\usepackage{algorithm}%
\usepackage{algorithmicx}%
\usepackage{algpseudocode}%
\usepackage{listings}%

\usepackage{mathrsfs}

\renewcommand{\div}{\operatorname{div}}

\def\R{\mathbb{R}}

\def\1B{{\bf  1}}

\newcommand{\NN}{\mathbb{N}}

\newcommand{\RR}{\mathbb{R}}

\def\EE{\mathbb{E}}

\newcommand\be{\begin{equation}}
\newcommand\ee{\end{equation}}
\newcommand\ba{\begin{array}}
\newcommand\ea{\end{array}}
\newcommand{\bean}{\begin{eqnarray*}}
\newcommand{\eean}{\end{eqnarray*}}

\newtheorem{lemma}{\textbf{Lemma}}[section]
\newtheorem{theorem}{\textbf{Theorem}}[section]

\numberwithin{equation}{section}

\newcommand{\du}{\mathbf{d}_1}
\newcommand{\into}{\int_\Omega}
\newcommand{\dbc}[1]{{#1}_{|x\in\partial\Omega}=0}
\newcommand{\nbc}[1]{{A(x)\nabla#1\cdot\nu(x)}_{|x\in\partial\Omega}=0}
\newcommand{\norm}[1]{\ensuremath{\left\Arrowvert #1 \right\Arrowvert}}
\newcommand{\norminf}[1]{\ensuremath{\left\Arrowvert #1 \right\Arrowvert_\infty}}
\newcommand{\conds}{\eqref{cauchy}, \eqref{dirichlet} or \eqref{neumann}}
\newcommand{\cc}{\eqref{cauchy}}
\newcommand{\cd}{\eqref{dirichlet}}
\newcommand{\cn}{\eqref{neumann}}
\newcommand{\ccs}{\eqref{cauchy} }
\newcommand{\cds}{\eqref{dirichlet} }
\newcommand{\cns}{\eqref{neumann} }

\newcommand{\amu}{_{\frac{1+\alpha}{2},1+\alpha}}

\newcommand{\inti}{\int_0^t\int_{\Omega}}

\newcommand{\trm}[1]{\sum\limits_{ij}\partial^2_{ij}\big(A_{ij}(x)#1\big)}

\DeclareMathOperator{\tr}{tr}

\begin{document}

\title[Forward-Forward Mean Field Games in mathematical modeling]{Forward-Forward Mean Field Games in mathematical modeling with application to opinion formation and voting models. }

\author*[1]{\fnm{Adriano} \sur{Festa}}\email{adriano.festa@polito.it}

\author[2]{\fnm{Simone} \sur{G\"ottlich}}\email{goettlich@uni-mannheim.de}
\equalcont{These authors contributed equally to this work.}

\author[3]{\fnm{Michele} \sur{Ricciardi}}\email{ricciardim@luiss.it}
\equalcont{These authors contributed equally to this work.}

\affil*[1]{\orgdiv{Dipartimento di Scienze Matematiche "Giuseppe Luigi Lagrange"}, \orgname{Politecnico di Torino}, \orgaddress{\street{Corso Duca degli Abruzzi, 24}, \city{Torino}, \postcode{10129},  \country{Italy}}}

\affil[2]{\orgdiv{Department of Mathematics}, \orgname{University of Mannheim}, \orgaddress{\street{A5-6}, \city{Mannheim}, \postcode{68131},  \country{Germany}}}

\affil[3]{\orgdiv{Department of Economics and Finance}, \orgname{Università Luiss Guido Carli}, \orgaddress{\street{Villa Blanc Via Nomentana, 216}, \city{Roma}, \postcode{00162}, \country{Italy}}}

\abstract{While the general theory for the terminal-initial value problem in mean-field games is widely used in many models of applied mathematics, the modeling potential of the corresponding forward-forward version is still under-considered. In this work, we discuss some features of the problem in a quite general setting and explain how it may be appropriate to model a system of players that have a complete knowledge of the past states of the system and are adapting to new information without any knowledge about the future. Then we show how forward-forward mean field games can be effectively used in mathematical models for opinion formation and other social phenomena.}




\maketitle

\section{Introduction}
Mean Field Games (MFGs) are models for large populations of competing rational agents who seek to optimize an individual objective function. They were introduced for the first time in 2006 in \cite{LasryLions07}. The most investigated form has a \textit{forward-backward} structure that can be sketched by the following system of partial differential equations: 
\begin{equation}\label{systemMFG}
\left\{
\begin{array}{ll}
-u_t(t,x)-\tr(A(x)D^2 u(t,x))+H(x,\nabla u(t,x))=F(t,x,m(t)),\\
m_t(t,x)-\sum\limits_{ij}\partial^2_{ij}(A_{ij}(x)m(t,x))-\mathrm{div}(H_p(x,\nabla u(t,x))m(t,x))=0.
\end{array}\right.
\end{equation}
Here, a density of individuals $m$ moves in the direction determined by the potential function $u$.  In the system above, $A$ is a diffusion matrix, $D^2$ is the Hessian, $\nabla$ is the usual space gradient and $H_p$ is the gradient w.r.t. the variable $p$ of the Hamiltonian operator $H(x,p)$. The notation $\partial^2_{ij}$ stands for the second-order derivatives with respect to the space coordinates $x_i$,$x_j$, finally, in $F$ the density $m(t):=m(t,\cdot)$, i.e. the coupling may be nonlocal.

To give a hint of the idea behind the model, the Hamilton-Jacobi equation (HJ) contained in the first line of the system solves, backward in time, an optimal control problem relative to the strategic choice of any individual. This information is used, forward in time, to move the density $m$ accordingly to the second equation of \eqref{systemMFG}.

The existence and uniqueness of equilibrium solutions of \eqref{systemMFG} can be found, for example, in \cite{Cardialaguet10,GomSau14}. Other researches on this subject are \cite{gomes2015regularity}, where strong solutions for parabolic problems were considered, or \cite{porretta2015weak}, for weak solutions of parabolic problems with Dirichlet and Neumann boundary conditions. In \cite{cardaliaguet2015weak} the authors discuss the existence of weak solutions for first-order MFGs.  The stationary case was also investigated in detail and it was first considered in  \cite{LasryLions07}.  Generally, the uniqueness of solution is obtained (both for stationary and time-dependent MFGs) via a monotonicity argument introduced in \cite{LasryLions07}.

\medskip
Another model closely related to our discussion is Hughes' model \cite{hughes,carlini2016DGA}. In this model, the optimization process of each player relies solely on information about the current state of the system, without any anticipation of future states. Consequently, the density evolution is associated with a stationary HJ that incorporates only the present configuration of the density. This model has been demonstrated to be suitable for effectively simulating scenarios in which agents may change their strategy due to unforeseen events. Notably, they are prone to changing their strategy instantaneously, with no consideration given to their past behavior. This modeling approach is particularly useful for simulating large crowd fluxes of pedestrians.

A version of a MFG system, in a forward-forward form, has been proposed in \cite{achdou2010mean}, as effective way to approximate the stationary state of the system, and in \cite{Diogo2018,gomes2016one}, in its one-dimensional form, mostly with the purpose of studying the qualitative properties of the system. 

\medskip

Our paper aims to investigate forward-forward Mean Field Games (FF-MFGs) from a modeling perspective. We seek to demonstrate that FF-MFGs are not merely effective tools for approximating stationary problems or theoretical curiosities; rather, they represent a potent and, in some cases, more suitable model for simulating collective processes. In this framework, agents are influenced by past states of the system, as they lack access to information about the future. Instead, their current strategic choices are significantly shaped by the collective past configurations. These distinctive features can be leveraged effectively to model collective motion phenomena, where factors such as tradition and stubbornness play a crucial role. Examples include opinion formation and voters' tendencies.

\medskip
We present a concise overview of the paper's content: In Section \ref{s1}, we formally derive a class of FF-MFGs, providing a stochastic interpretation in terms of optimal control for stochastic differential equations. Section \ref{s2} contains general results on existence and uniqueness, demonstrating their broad applicability. We emphasize the generality of these findings, showcasing proofs for solutions in the Cauchy problem and cases involving Dirichlet and Neumann boundary conditions.

While we do not explore periodic solutions in this work, specifically when the state space $\Omega=\mathbb T^d$, it can be demonstrated more straightforwardly compared to other cases. In Section \ref{s3}, we introduce and discuss two FF-MFG models for social processes, accompanied by various numerical simulations that practically illustrate their features. The numerical tools employed throughout the paper are briefly outlined in Appendix \ref{app:num.tools}.

\section{Derivation and interpretation of the evolution process}\label{s1}

Let us denote by $X_s \in \R^n$ the quantity of interest (e.g. opinion, ideological position on a political spectrum, etc.) at time $s$, which follows the stochastic differential equation (SDE)
\begin{equation}\label{dynamicsX} 
\left\{\begin{array}{ll}
d  X_s= \beta(s,X_s) d  s +  \sqrt{2}\sigma(X_s)\,dB_s, &  s\in (t,T), \\
X_t=x, & 
\end{array}\right.
\end{equation}
for some $x \in \R^n$. Here, $\sigma(X_s)$ is a $n\times n$ matrix, and $B_s$ is a standard $n$-dimensional Brownian motion.  We define the quantity $m$ as the probability density of players moving accordingly to the SDE \eqref{dynamicsX}
and initial distribution $m(0,x)= m_0(x) \in \mathcal{P}(\R^n)$. Classical arguments  in diffusion theory imply that the measures $m$ of the initial distribution $m_0$ will be characterized by the dynamics given by the Fokker-Planck (FP) equation 
\begin{equation}\label{FPEquation}
\left\{\begin{array}{ll}
m_t(t,x) -\sum\limits_{ij}\partial^2_{ij}(A_{ij}(x)m(t,x))+\mbox{div}\big(\beta(t,x)\, m(t,x)  \big) = 0,\\
m(0,x)=m_0(x)&
\end{array}\right.
\end{equation}
with $A(x):=\sigma(x)\sigma^\mathsf{T}(x)$. The latter has to be understood in the weak sense \cite{Cardialaguet10}. 

The crux of the model lies in the selection of the drift, denoted as $\beta$. We make the assumption that the players \emph{are uninformed} about future states of the system (i.e., $m(t, x)$ for $t\in[t, T]$). However, they assume that certain tendencies from the past may reemerge in the future. This scenario is particularly relevant in domains like political opinions or voting intentions, where fluctuations around specific poles of attraction often manifest. In such instances, recent past events exert a more pronounced influence on individual choices, whereas distant past events, wield a comparatively limited impact.

With the latter in mind, for any fixed $\bar t\in(0,T]$ consider the optimal control problem
\begin{multline}
\label{representacionestocastica}
v(\tau,x)= \inf_{\alpha} \;\left\{ \mathbb E \Big[\int_{\tau}^{2\bar t}e^{-\lambda(s-\tau)}\left[\ell(s,Y^{x, \tau}[\alpha](s), \alpha(s,Y^{x,\tau}[\alpha](s)))\right.\right.\\
\left.\left.+ F(s,Y^{x,\tau}[\alpha](s),m(2\bar t-s)) ]d  s +e^{-\lambda (2\bar t-\tau)}u_0(Y^{x,\tau}[\alpha](2\bar t))\right]\right\},
\end{multline}
where $Y^{x,\tau}[\alpha](s):=Y_s$ with $s \in [\tau,2\bar t]$ is assumed to be the unique strong solution to 
\begin{equation}\label{dynamicsY} 
\left\{\begin{array}{ll}
d  Y_s= \alpha(s,Y_s) d  s +  \sqrt{2}\sigma(Y_s)\,dB_s, &  s\in (\tau,2\bar t), \\
Y_\tau=x, & 
\end{array}\right.
\end{equation}
$\ell \colon \R^n\times \R^n \to \R$ and $F\colon \R^n\times\mathcal{P}(\R^n) \to \R$ are running cost functions, and $\lambda\in [0,+\infty)$ is a time discount factor. Defining a Hamiltonian function $\tilde H$ as 
\begin{equation}\label{hamiltonian}
\tilde H(\tau,x,p):=\max_{\alpha\in \R^n}\left\{-\alpha \cdot p-\ell(\tau,x,\alpha)\right\},
\end{equation}
the optimal solution of the optimization problem in \eqref{representacionestocastica} is given in feedback form by $\alpha(\tau,x)= -\tilde H_p(\tau,x, \nabla v(\tau,x))$.
Moreover, a classic Dynamic Programming Principle (see \cite{BCD97}) applies to the value function $v(s,x)$ in \eqref{representacionestocastica} bringing us a differential representation formula for $v(s,x)$ in HJ form as 
\begin{equation*} 
-v_\tau(\tau,x)-\mbox{tr} (A(x) D^2v(\tau,x))+\tilde H(\tau,x,\nabla v(\tau,x))+\lambda v(\tau,x)=\tilde F(\tau,x,m(2\bar t-\tau)),
\end{equation*}
for $ \tau\in[\bar t,2\bar t)$, provided with the terminal condition $v(2\bar t,x)=u_0(x)$. 

We can interpret the function $v(\tau, x)$ as the value function for an agent who engages in an optimal control problem at every instant $\bar t \in [0, T]$ within the time interval $[\bar t, 2\bar t]$. In the absence of knowledge about the future states of the system, the agent replaces them by mirroring the past states of the time interval $[0, \bar t]$, which are assumed to be known. The discount coefficient $\lambda$ gauges the significance of the past in this decision-making process. Therefore, we anticipate that as $\lambda$ approaches $+\infty$, the models will converge towards an instantaneous response of the system, where trajectory optimization is not performed.

Considering the special nature of the coupling between $v$ and $m$ it is convenient to reverse the time evolution: if we define $v(\tau,x)=u(2\bar t-\tau,x)$ and we move the time interval in $[0,\bar t]$ (with the substitution $t=2\bar t-\tau$) we obtain 
\begin{equation*} 
\left\{\begin{array}{ll}
u_{t}(t,x)-\mbox{tr} (A(x) D^2 u(t,x))+ H(t,x,\nabla u(t,x))+\lambda u(t,x)= F(t,x,m(t)), \quad t\in(0,\bar t] \\
u(0,x)=u_0(x)\,,
\end{array}\right.
\end{equation*}
which holds for any $\bar t\in[0,T]$ and therefore in particular for $\bar t=T$. Here, $ H$ and $ F$ are defined in the following way:
$$
 H(t,x,p)=\tilde H(2\bar t-t,x,p)\,,\qquad  F(t,x,r)=\tilde F(2\bar t-t,x,r)\,.
$$
We then set $\beta(t,x)=- H_p(t,x,\nabla u(t,x))$: the latter equation, coupled with \eqref{FPEquation}, gives a Forward-Forward MFG system.

\medskip

Clearly, the motion of the density and the optimal control do not follow the same trajectories. Therefore, similar to the Hughes' model, this framework cannot be accurately labeled as a \emph{game}. This is a critical aspect that enables our model to describe potential changes in strategy and less rational behaviors among the players, as in the case of high stubbornness and ideological bias. However, we maintain the term FF-MFGs for consistency with the existing literature.

\medskip
The model can be defined in a bounded open subset of $\mathbb{R}^n$; in such cases, additional boundary conditions are necessary. A stochastic interpretation can be provided for Dirichlet and Neumann boundary conditions: for Dirichlet boundary conditions, the player's trajectory must be halted upon reaching the boundary. In the case of Neumann boundary conditions, a reflecting process at the boundary is incorporated into the stochastic differential equation \eqref{dynamicsX}, as detailed in \cite{sznitman}. However, we do not delve into these interpretations here, as they involve a readaptation of the backward-forward MFG developed in \cite{MED,convergenceneumann}.

\section{Some results about Forward-Forward Mean Field Games}\label{s2}
The forward-forward model was introduced  in \cite{achdou2010mean} to  approximate stationary MFGs.
The key insight is that the parabolicity in (\ref{MFG}) should imply the long-time convergence to a stationary solution. 

\medskip
In this section, we want to prove existence and uniqueness of classical solutions for the forward-forward MFGs. In particular we discuss the existence and the uniqueness of the FF-MFG general framework which includes the models that we are going to discuss in the rest of the paper. We consider the  {\em forward-forward MFG} problem described by the system
\begin{equation}
\label{MFG}
\left\{
\begin{array}{ll}
	\displaystyle u_t-\mathrm{tr}(A(x)D^2 u)+H(t,x,\nabla u)+\lambda u=F(t,x,m)\,, & (t,x)\in (0,T)\times\Omega\\
	\displaystyle m_t-\trm{m}-\mathrm{div}\big(mH_p(t,x,\nabla u)\big)=0\,, & (t,x)\in (0,T)\times\Omega\\
	\displaystyle u(0,x)=u_0(x)\,,\qquad m(0,x)=m_0(x)\,, & x\in\Omega\,.
\end{array}\right.
\end{equation}
Here, $\Omega\subseteq\R^n$ is an open domain. Note that in this section we consider the case $\lambda\ge0$. This is done just to simplify the presentation. The presence of a negative smoothing term is not a problem and the same results hold, since the function $z(t,x):=e^{(|\lambda|-\lambda) t}u(t,x)$ solves the same Hamilton-Jacobi equation with $\lambda$ replaced by $|\lambda|$. We distinguish between three cases, giving rise to three different kind of problems. Namely, either
\begin{equation}\label{cauchy}
\Omega=\R^n\tag{C}\,,\mbox{ no boundary conditions prescribed,}
\end{equation}
which is known as \emph{Cauchy problem}, or we require $\Omega$ to be a bounded domain (the regularity of such domain is stated in the hypotheses) complementing the system with homogeneous \emph{Dirichlet} or \emph{Neumann} boundary conditions, i.e.
\begin{equation}\label{dirichlet}
u(t,x)_{|x\in\partial\Omega}=0\,,\qquad m(t,x)_{|x\in\partial\Omega}=0\,,\tag{D}
\end{equation}
or, for $x\in\partial\Omega$,
\begin{equation}\label{neumann}
A(x)\nabla u(t,x)\cdot\nu(x)=0\,,\qquad \Big[\big(\sum\limits_j\partial_j(A_{ij}(x)m)\big)_i+m H_p(t,x,\nabla u)\Big]\cdot\nu_{|\partial\Omega}=0\,.\tag{N}
\end{equation}

Recall that in \eqref{hamiltonian} we have $H(t,x,p)=\max\limits_{\alpha\in\R^n}\big\{-\alpha\cdot p - \ell(2T-t,x,\alpha)\big\}$. Nevertheless, the results given in this section are more general and may include other kind of Hamiltonian.

The existence and the long-time convergence of the forward-forward model have not been addressed, except in a few cases, see  \cite{gomes2015regularity}
and  \cite{gueant2011mean}. In \cite{gueant2011mean}, the  forward-forward problem was examined in the context of eductive stability of stationary MFGs with a logarithmic coupling.  In \cite{gomes2015regularity}, the existence and regularity of solutions for the forward-forward, uniformly parabolic MFGs with subquadratic Hamiltonians was proven.
A general result for existence and uniqueness of solutions for systems of parabolic equations was given, just with Dirichlet boundary conditions, in \cite[Theorem VII.7.1]{ladyzenskaja1968linear}. Hence, the results provided in this section can be considered as a generalization of that (since we include also the cases \eqref{neumann} and \eqref{cauchy}) of that. Moreover, the proof was not given there, but the techniques used (fixed-point method) are the same we use in this paper.

Except for these cases, the question of existence and regularity is open  in all other regimes.  In the case of forward-forward MFGs without viscosity, these questions are particularly challenging. Moreover, the long-time convergence has not been established even in the parabolic case. Nevertheless, numerical results in \cite{achdou2017mean} indicate that convergence holds and that the forward-forward model approximates well stationary solutions.

\medskip 
Let $\mathcal{P}(\Omega)$ (resp. $\mathcal P^{sub}(\Omega)$) be the set of probability measures (resp. sub-probability measures) on $\Omega$ with finite first moments. Then, $\mathcal{P}(\Omega)$ is a topological space endowed with the weak*-convergence and metrizable with respect to the \emph{Wasserstein distance}, which is defined as the distance between two probability measures $m_1,m_2\in \mathcal{P}(\Omega)$ as
\[
\du(m_1,m_2) =\inf_{(X,Y)\in \Pi(m_1, m_2)} \EE\left[|X-Y|\right]
\]
Here $\Pi(m_1, m_2)$ is the set of all joint probability distributions $(X,Y)$ with $\mathscr L(X_1)=m_1$, $\mathscr L(X_2)=m_2$. An equivalent definition is the following one:
\begin{align*}
\du(m_1,m_2)&:=\sup_{g\in Lip_1(\Omega)} \left\{\into g(x)d(m_1-m_2)(x)\right\},
\end{align*}
where with $Lip_1(\Omega)$ we mean the space of Lipschitz continuous functions with Lipschitz constant equal to one. The equivalence of the two definitions was given in \cite{Cardialaguet10}.

The same definition applies in case of \eqref{dirichlet}, with $\mathcal P(\Omega)$ replaced by $\mathcal P^{sub}(\Omega)$ and test function $g$ with $Lip(g)\le 1$ and $\norminf{g}\le 1$, see \cite{Piccoli2}. This is called \emph{generalized Wasserstein distance}.

We observe that, in case of \eqref{neumann} and \eqref{cauchy}, we can restrict to test functions $g$ with $g(x_0)=0$, for a fixed $x_0\in\Omega$. In particular, in case of Neumann condition we have $\norminf{g}\le C$, where $C=diam(\Omega)$, and in case of the Cauchy problem we have $|g(x)|\le|x|$.

Henceforth, when there is no possibility of mistake, we use the notation $\mathbf{d_1}$ to refer to both Wasserstein distance and generalized Wasserstein distance.

\medskip 
We start giving the basic assumptions that we need for an existence and uniqueness result.
\paragraph{Hypotheses:}
Let $\gamma\in(0,1)$ and $C>0$. We assume that, in the cases \eqref{dirichlet} and \eqref{neumann}, $\Omega$ is a bounded domain with $\mathcal C^{2+\gamma}$ boundary. Moreover, we suppose that
\begin{enumerate}
	\item [(H0)]\label{H0} $u_0\in\mathcal{C}^{2+\gamma}(\Omega)$, $m_0\in\mathcal{P}(\Omega)$, or $\mathcal{P}^{sub}(\Omega)$ in the case \eqref{dirichlet}, with a $\mathcal C^{2+\gamma}$ density, called also $m_0$. Moreover $A(\cdot)\in\mathcal{C}^{2}(\overline\Omega)$, $A=\sigma\sigma^\mathsf{T}$ with $\sigma\in Lip(\Omega;\R^{n\times n})$ and for certain $\mu_2>\mu_1>0$ independent of $x$
	$$
	\mu_1|\xi|^2\le \langle A(x)\xi,\xi\rangle\le \mu_2|\xi|^2\qquad\forall\,\xi\in\R^n;
	$$
	\item [(H1)]\label{H1} $H:[0,T]\times\Omega\times\R^n\to\R$, $F:[0,T]\times\Omega\times\mathcal P(\Omega)\to\R$ (or $F:[0,T]\times\Omega\times\mathcal P^{sub}(\Omega)\to\R$ in the case \eqref{dirichlet}) are continuous functions and $\mathcal{C}^{\frac\gamma2,\gamma}$ in $(t,x)$ variable, with $H$ differentiable and convex in the last variable, and $H(\cdot,\cdot,0)$ is globally bounded. Moreover, the growth of $H$ is at most quadratic:
	$$
	|H(t,x,p)|\le C(1+|p|^2)\,,\qquad |H_p(t,x,p)|\le C(1+|p|)\,,\qquad |H_{pp}(t,x,p)|\le C\,;
	$$
	\item [(H2)]\label{H2} For $\mu_1,\mu_2\in\mathcal{P}(\Omega)$ or $\mathcal P^{sub}(\Omega)$, the following condition hold:
	\begin{equation*}
		\begin{split}
  & |F(t,x,\mu_1)-F(t,x,\mu_2)|\le C\du(\mu_1,\mu_2)\,;
		\end{split}
	\end{equation*}
	\item [(H3)]\label{H3} In case of boundary conditions, we require compatibility conditions for the initial data: namely, in the setting \eqref{dirichlet}, we require
	$$
	\dbc{u_0(x)}\,,\qquad \dbc{m_0(x)}\,,
	$$
	whereas in the setting \eqref{neumann}
	$$
	\nbc{u_0(x)}\,,\qquad \big[A(x)\nabla m_0+ m_0H_p(0,x,\nabla u_0)\big]\cdot\nu(x)=0\,.
	$$
	In the setting \cc, we require instead a finite first order moment for $m_0$, namely
	$$
	\int_{\R^n}|x|\,m_0(dx)\le C\,.
	$$
\end{enumerate}

\medskip

From now on, $C$ indicates a non-negative constant, that can change from line to line.
\begin{theorem} \label{T1}
Let Hypotheses $(H1)-(H3)$ hold true. Then there exists a unique solution $(u,m)\in\mathcal{C}^{1+\frac\gamma2,2+\gamma}([0,T]\times \overline\Omega)\times\mathcal C^{1+\frac\gamma2,2+\gamma}([0,T]\times \overline\Omega)$ for the FF-MFG system \eqref{MFG}.
\end{theorem}
\begin{proof}
For $M>0$ which will be chosen later, we consider the following set:
$$
X=\left\{ \mu\in\mathcal C([0,T];V)\quad|\quad d_V(\mu(t),\mu(s))\le M|t-s|^{\frac12} \right\}\,,
$$
where in case of \eqref{cauchy} or \eqref{neumann}, $V=\mathcal P(\Omega)$ and $d_V=\du$, otherwise in case of \eqref{dirichlet} $V$ is the dual of $\left\{\phi\in\mathcal C^1(\overline\Omega)\,|\,\dbc{\phi}\right\}$, with $d_V$ the classical distance induced by the duality norm:
$$
d_V(\mu,\theta)=\norm{\mu-\theta}_V=\sup\limits_{\substack{\norm{\phi}_{\mathcal C^1}\le 1\\\phi_{|\partial\Omega=0}}}\langle \mu-\theta,\phi\rangle\,.
$$
Note that $X$ is closed and convex in the topology inherited by $\mathcal C([0,T];V)$. In order to apply Schauder's Theorem, we define a functional $\Phi:X\to X$ in the following way: for $\mu\in X$, we take $u=u_\mu$ as the solution of
\begin{equation*}
	\begin{cases}
		u_t-\tr(A(x)D^2 u)+H(t,x,\nabla u)+\lambda u=F(t,x,\mu(t))\,,\\
		u(0)=u_0\,,
	\end{cases}
\end{equation*}
with conditions \conds.

In all the three cases, a solution $u$ is unique and regular. In case of boundary conditions \eqref{dirichlet} and \cn, from the comparison principle we have $\norminf{u}$ uniformly bounded. Then from \emph{Theorem V.4.1} and \emph{Lemma VI.3.1} (resp. \emph{Theorem V.7.2}) of \cite{ladyzenskaja1968linear}, we have $u$ globally Lipschitz in the space variable $x$. Hence, we can linearize the problem of $u$, which turns out to be the solution of
\begin{equation*}
	\begin{cases}
		u_t-\tr(A(x)D^2 u)+V(t,x)\cdot\nabla u+\lambda u=F(t,x,\mu(t))+H(t,x,0)\,,\\
		u(0)=u_0\,,
	\end{cases}
\end{equation*}
with conditions \eqref{dirichlet} or \cn, and with
$$
V(t,x):=\int_0^1 H_p(t,x,s Du(t,x))\,ds\,.
$$
Then we can apply the Corollary of \emph{Theorem IV.9.1} of \cite{ladyzenskaja1968linear} to obtain, for all $0\le\gamma<1$,
\begin{equation*}
	\begin{split}
		\norm{u}_{\frac{1+\gamma}2,1+\gamma}\le C\Big(\norminf{F}+\norminf{H(\cdot,\cdot,0)}+\norm{u_0}_{2}\Big)\,.
	\end{split}
\end{equation*}
For the Cauchy problem \cc, we consider, for $k\in\NN$, the solution $u^k$ of the following problem
\begin{equation*}
	\begin{cases}
		u^k_t-\mathrm{tr}(A(x)D^2 u^k)+H^k(t,x,\nabla u^k)+\lambda u^k=F(t,x,\mu(t))\,,\\
		u^k(0)=u_0\,,
	\end{cases}
\end{equation*}
where
$$
H^k(t,x,p)=\frac{kH(t,x,p)}{k+|p|^2}
$$
With this choice, $H^k(\cdot,\cdot,p)$, $H_p^k(\cdot,\cdot, p)\in L^{\infty}([0,T]\times\R^n)$ and $H^k\to H$ pointwise. A straightforward application of the comparison principle (see \cite{crandall}) implies that $u^k(t,\cdot)$ is bounded in $W^{1,\infty}$ uniformly in $t$ and $k$, and so the same estimate holds for $u$. Then we can apply \emph{Theorem V.8.1} of \cite{ladyzenskaja1968linear} and obtain $u\in\mathcal C^{1+\frac\gamma2,2+\gamma}([0,T]\times\R^n)$. Finally, we observe that $u$ solves
$$
\begin{cases}
	u_t-\tr(A(x)D^2u)=G(t,x)\,,\\
	u(0)=u_0\,,
\end{cases}
$$
with $G(t,x)=F(t,x,\mu(t))-H(t,x,\nabla u)-\lambda u\in L^\infty$. Hence, from \emph{Theorem 5.1.2} of \cite{lunardi} and the hypotheses on $H$ we have, for all $0\le\gamma<1$,
$$
\norm{u}_{\frac{1+\gamma}2,1+\gamma}\le C\Big(1+\norminf{F}+\norminf{\nabla u}^2+\norm{u_0}_{2}\Big)\le C\,,
$$
where $C$ does not depend on $\mu$.

Moreover, from this estimate we get $G$ bounded in $\mathcal C^{0,\gamma}([0,T]\times\R^n)$, uniformly in $\mu$. Hence, from \emph{Theorem 5.1.9} of \cite{lunardi} we get
\begin{equation}\label{eq:stima2}
	\begin{split}
		\norm{u}_{1,2+\gamma}\le C\left(\norm{u_0}_{2+\gamma}+\norm{G}_{0,\gamma}\right)\le C\,,
	\end{split}
\end{equation}
where $C$ does not depend on $\mu$.

Then we choose $\Phi(\mu)=m$, where $m=m_\mu$ is a solution, in the distributional sense, of the FP equation
\begin{equation}\label{eq:fp}
	\begin{cases}
		m_t-\trm{m}-\mathrm{div}\big(mH_p(t,x,\nabla u)\big)=0\,,\\
		m(0)=m_0(x)\,,
	\end{cases}
\end{equation}
with conditions \conds.

For the distributional formulation, $m$ satisfies, for all $0<s<t<T$,
\begin{equation*}
	\into m(t,x)\psi(x)\,dx+\int_s^t\into m(r,x)f(r,x)\,dxdr=\into m(s,x)\phi(s,x)\,dx\,,
\end{equation*}
where $f\in L^\infty$ and $\phi\in L^\infty$ satisfies
\begin{equation}\label{eq:phi}
	\begin{cases}
		-\phi_t-\mathrm{tr}(A(x)D^2\phi)+H_p(s,x,\nabla u)\cdot\nabla\phi=f\,,\\
		\phi(t)=\psi\,,
	\end{cases}
\end{equation}
with conditions \conds. We note that the results obtained for the HJ equation can be applied also for the backward equation \eqref{eq:phi}, with the change of variable $(t,x)\to (T-t,x)$.

Taking $f=0$ and rearranging, we find
\begin{equation*}
	\begin{split}
		\into \big(m(t,x)-m(s,x)\big)\psi(x)\,dx=\into m(s,x)(\phi(s,x)-\phi(t,x))\,dx
	\end{split}
\end{equation*}

In case of conditions \ccs or \cn, we choose $Lip(\psi)=1$. As already said, for \cns we can also take $\psi\in W^{1,\infty}$ and thanks to \emph{Lemma 3.2} of \cite{MEN} we have
\begin{equation}\label{eq:holdert}
	|\phi(t,x)-\phi(s,x)|\le C|t-s|^{\frac12}\,,
\end{equation}
which implies, for a $C$ not depending on $\mu$,
\begin{equation}\label{uff}
	\du(m(t),m(s))\le C|t-s|^{\frac12}\,.
\end{equation}
We argue in a similar way for the case \cc. Here, we have $\nabla\phi$ uniformly bounded in $L^\infty$. Hence, $\phi$ solves
\begin{equation*}
	\begin{cases}
		-\phi_t-\mathrm{tr}(A(x)D^2\phi)=f-H_p(s,x,\nabla u)\cdot\nabla\phi\in L^\infty\,,\\
		\phi(t)=\psi\,,
	\end{cases}
\end{equation*}
and applying again \emph{Theorem 5.1.2} of \cite{lunardi}, \eqref{eq:holdert} remains true, which implies once again \eqref{uff}.

As regards Dirichlet boundary conditions \cd, we take $\psi\in \mathcal C^1(\Omega)$ with $\dbc{\psi}$ and $\norm{\psi}_1\le 1$. From \emph{Lemma 4} of \cite{MED} we get
$$
\norm{\phi}_{\frac12,1}\le C\norm{\psi}_1\implies d_V(m(t),m(s))\le C|t-s|^{\frac12}.
$$
Hence, if we choose $M$ sufficiently large, we have in all cases $m\in X$, and $\Phi:X\to X$ is well-defined.

To apply Schauder's theorem, we have to check that $\Phi(X)$ is relatively compact in $\mathcal C([0,T];V)$ and continuous. For the relatively compactness, let $\{\mu^k\}_k\subset X$, and let $u^k$ and $m^k$ be the solutions of the HJ and the FP equation related to $\mu^k$.

As already seen, we have $\|u^k\|\amu\le C$, where $C$ does not depend on $k$. Hence, up to subsequences, we can apply Ascoli-Arzelà's Theorem and, for a certain $u\in\mathcal C^{\frac{1+\gamma}2,2+\gamma}$,
$$
u^k\to u\qquad\mbox{in }\mathcal C^{0,1}([0,T]\times\Omega)\,,
$$
at least for boundary conditions \cds or \cn. In those cases, take $\phi^k$ and $\phi$ the solution of \eqref{eq:phi} related to $u^k$ and $u$, with $\psi\in W^{1,\infty}$, satisfying $\dbc{\psi}$ in case of \cd. Then, arguing as before, $\phi^k$ is uniformly bounded in $\mathcal C^{\frac{1+\gamma}2,1+\gamma}$ and, up to subsequences, $\phi^k\to\phi$ in $\mathcal C([0,T]\times\Omega)$. Then, subtracting the distributional formulations of $m^k$ and $m$ with test functions $\phi^k$ and $\phi$ we get
$$
\into \psi(x)(m^k(t,x)-m(t,x))\,dx=\into (\phi^k(0,x)-\phi(0,x))\,m_0(dx)\to0\,,
$$
which implies $m^k\to m$ in $\mathcal C([0,T];V)$.

In case of boundary conditions \cc, the convergence of $u^k$ towards $u$ is true only in $\mathcal C^{0,2}_{loc}([0,T]\times\Omega)$, thanks to \eqref{eq:stima2}, with $u\in\mathcal C^{1,2+\gamma}([0,T]\times\Omega)$. In this case we note that $m^k$ is the density of the following process
\begin{equation*}
	\begin{cases}
		dX_t^k=b^k(t,X_t^k)\,dt+\sqrt 2\sigma(X_t^k)\,dB_t\,,\\
		X_0^k=X_0\,,
	\end{cases}
\end{equation*}
where $X_0$ is any fixed process with $\mathscr L(X_0)=m_0$, and
$$
b^k(t,x)=-H_p(t,x,Du^k(t,x))
\,,\qquad \sigma(x)\sigma^\mathsf{T}(x)=A(x)\,.
$$
In the same way we define the drift $b$ and the process $X_t$ related to $m$, solution of the FP equation \eqref{eq:fp}. Then, we have
\begin{multline*}
		\EE\left[|X_t^k-X_t|^2\right]\le 2\,\EE\left[\int_0^t \left|b^k(s,X_s^k)-b^k(s,X_s)\right|^2\,ds\right]\\
		+4\,\EE\left[\left|\int_0^t\big(\sigma(X_s^k)-\sigma(X_s)\big)\,dB_s\right|^2\right]\,.
\end{multline*}
The second integral is bounded using the It\^{o} isometry and the Lipschitz bound on $\sigma$:
$$
\EE\left[\left|\int_0^t\big(\sigma(X_s^k)-\sigma(X_s)\big)\,dB_s\right|^2\right]\le C\,\EE\left[\int_0^t|X_s^k-X_s|^2\,ds\right]\,.
$$
For the first term we get
\begin{equation*}
	\begin{split}
		&\EE\left[\int_0^t \left|b^k(s,X_s^k)-b^k(s,X_s)\right|^2\,ds\right]\le C\,\EE\left[\int_0^t|X_s^k-X_s|^2\,ds\right]\\
		+C\,&\EE\left[\int_0^t \big|H_p(s,X^k_s,\nabla u^k(s,X_s^k))-H_p(s,X_s,\nabla u(s,X_s))\big|^2\,ds\right]\,.
	\end{split}
\end{equation*}
For the last integral we use the following strategy:
\begin{equation*}
	\begin{split}
		&\EE\left[\int_0^t \big|H_p(s,X^k_s,\nabla u^k(s,X_s^k))-H_p(s,X_s,\nabla u(s,X_s))\big|^2\,ds\right]\\
		\le 2\,&\EE\left[\int_0^t \big|H_p(s,X^k_s,\nabla u^k(s,X_s^k))-H_p(s,X^k_s,\nabla u(s,X^k_s))\big|^2\,ds\right]\\
		+2\,&\EE\left[\int_0^t \big|H_p(s,X^k_s,\nabla u(s,X_s^k))-H_p(s,X_s,\nabla u(s,X_s))\big|^2\,ds\right]\\
		\le C\,&\int_0^t\int_{\R^n}|\nabla(u^k-u)|^2\,m^k(s,x)\,dxds+C\,\EE\left[\int_0^t|X_s^k-X_s|^2\,ds\right]\,.\\
	\end{split}
\end{equation*}
This means
\begin{equation}\label{eq:pregronwall}
	\begin{split}
		\EE\left[|X_t^k-X_t|^2\right]\le C\int_0^t\int_{\R^n}|\nabla(u^k-u)|^2\,m^k(s,x)\,dxds+C\,\EE\left[\int_0^t|X_s^k-X_s|^2\,ds\right]\,.
	\end{split}
\end{equation}
The space-time integral is handled in the following way: for a compact $E\subset\R^n$, we have
	\begin{multline*}
		\int_0^t\int_{\R^n}|\nabla(u^k-u)|^2\,m^k\,dxds=\int_0^t\int_{E}|\nabla(u^k-u)|^2\,m^k\,dxds\\+\int_0^t\int_{E^c}|\nabla(u^k-u)|^2\,m^k\,dxds
		\le\norm{u^k-u}^2_{\mathcal C^{0,1}([0,T]\times E)}+C\norm{m^k}_{L^1([0,T]\times E^c)}\,.
		\end{multline*}
The first term goes to $0$ for each compact $E\subset\R^n$. For the second term, we note that $m^k$ has finite first order moments, uniformly bounded in $k$. Actually, if we take the solution $\phi$ of \eqref{eq:phi} with $f=0$ and $\psi=|x|$, from standard regularity results (see \cite{Priola}) we have
$$
\norminf{\phi}\le C(1+|x|)\,.
$$
Using hypothesis \emph{(H3)}, we obtain in the weak formulation of $m^k$
$$
\int_{\R^d}|x| m^k(t,x)\,dx=\int_{\R^d}\phi(0,x)\,m_0(dx)\le C\left(1+\int_{\R^d}|x|m_0(dx)\right)\le C\,.
$$
This implies that $\norm{m^k}_{L^1([0,T]\times E^c)}$ can be chosen arbitrarily small, if we take $E$ sufficiently large. This implies that
$$
\int_0^t\int_{\R^n}|\nabla (u^k-u)|^2\,m^k\,dxds=\omega(k)\,,\qquad\mbox{where }\lim\limits_{k\to+\infty}\omega(k)=0\,.
$$
Coming back to \eqref{eq:pregronwall} and using Gronwall's inequality, we finally obtain
	\begin{multline*}
\EE\left[|X_t^k-X_t|^2\right]\le C\,\EE\left[\int_0^t|X_s^k-X_s|^2\,ds\right]+\omega(k)\implies\\
\EE\left[|X_t^k-X_t|^2\right]\le e^{Ct}\omega(k)\to0\,.
	\end{multline*}
Since
$$
\du(m^k(t),m(t))=\sup\limits_{Lip(\phi)= 1}\EE\left[\phi(X_t^k)-\phi(X_t)\right]\le C\sqrt{\EE[|X_t^k-X_t|^2]}\,,
$$
we have proved that $m^k\to m$ in $\mathcal C([0,T];V)$ (up to subsequences) for all the boundary conditions \conds.

The continuity argument can be proven in the same way of the compactness. This concludes the existence result.

Observe that, since $u$ is smooth, we can split the divergence term in \eqref{eq:fp}, obtaining for $m$ the equation
$$
\begin{cases}
	m_t-\trm{m}-\mathrm{div}(H_p(t,x,\nabla u))m-H_p(t,x,\nabla u)\nabla m\,,\\
	m(0)=m_0\,,
\end{cases}
$$
with conditions \eqref{cauchy}, \eqref{dirichlet} or \eqref{neumann}. This implies, arguing as in the HJ equation, $m\in\mathcal C^{1+\frac\gamma2,2+\gamma}$.\\

For the uniqueness, let $(u_1,m_1)$ and $(u_2,m_2)$ be two solutions of \eqref{MFG}, and set $(z,\mu):=(u_1-u_2,m_1-m_2)$. Then, called
$$
V(t,x):=\int_0^1 H_p\big(t,x,s\nabla u_1(t,x)+(1-s)\nabla u_2(t,x)\big)\,ds\,,
$$
the couple $(z,\mu)$ solves the following system:
\begin{equation*}
	\begin{cases}
		z_t-\tr(A(x)D^2 z)+V(t,x)\cdot\nabla z+\lambda z=F(t,x,m_1(t))-F(t,x,m_2(t))\,,\\
		\mu_t-\trm{\mu}-\mathrm{div}(\mu H_p(t,x,\nabla u_1))=\mathrm{div}\big(m_2(H_p(t,x,\nabla u_1)-H_p(t,x,\nabla u_2))\big)\,,\\
		z(0)=0\,,\qquad \mu(0)=0\,,
	\end{cases}
\end{equation*}
coupled, in the case of Dirichlet boundary conditions, with conditions \eqref{dirichlet} for $(z,\mu)$. In the case of Neumann boundary conditions, we have for the couple $(u,m)$ and for $x\in\partial\Omega$ the conditions
$$
A(x)\nabla z(t,x)\cdot \nu(x)=0
$$
$$
\Big[\big(\sum\limits_j\partial_j(A_{ij}(x)\mu)\big)_i+\mu H_p(t,x,\nabla u_1)+m_2\big(H_p(t,x,\nabla u_2)-H_p(t,x,\nabla u_1)\big)\Big]\cdot\nu_{|\partial\Omega}=0\,.
$$
We start with the two cases \eqref{dirichlet} and \eqref{neumann}. Multiplying by $\mu$ the Fokker-Planck equation of $\mu$, we obtain
\begin{equation*}
	\begin{split}
		\frac12\into \mu(t)^2\,dx&+\frac{\mu_1}2\int_0^t\into|\nabla\mu|^2\,dx ds+\inti \mu H_p(t,x,\nabla u_1)\nabla\mu\,dx ds\\
		&\le\inti m_2\big(H_p(t,x,\nabla u_2)-H_p(t,x,\nabla u_1)\big)\nabla\mu\,dx ds.
	\end{split}
\end{equation*}
Since $\norminf{\nabla u_1}\le C$, we get from Young's inequality
$$
\inti \mu H_p(t,x,\nabla u_1)\nabla\mu\,dx ds\ge-\frac{\mu_1} 4\inti|\nabla\mu|^2\,dx ds-C\inti \mu^2\,dx ds\,.
$$
In the same way, since $\norminf{m_2}\le C$ and $H_p$ is Lispchitz in the last variable, we get
\begin{multline*}
\inti m_2\big(H_p(t,x,\nabla u_2)-H_p(t,x,\nabla u_1)\big)\nabla\mu\,dx ds\\
\le\frac{\mu_1} 8\inti|\nabla\mu|^2\,dx ds+C\inti|\nabla z|^2\,m_2\,dx ds.
\end{multline*}
This means
$$
\frac12\into\mu(t)^2\,dx+\frac{\mu_1}8\inti|\nabla\mu|^2\,dx ds\le C\inti\mu^2\,dx ds+C\int_0^t\norm{\nabla z}^2_{L^\infty([0,s]\times\Omega)}\, ds
$$
which implies from Gronwall's Lemma
\begin{equation}\label{est_mu}
	\begin{split}
		\sup\limits_{s\in[0,t]}\into\mu(s)^2\,dx+\inti|\nabla\mu|^2\,dx ds\le C\int_0^t\norm{\nabla z}^2_{L^\infty([0,s]\times\Omega)}\, ds\,.
	\end{split}
\end{equation}
To estimate the last integral, we study the equation of $z$. Since $V\in L^\infty$ and, for all $s\in[0,T]$,
$$
\norm{F(t,x,m_1(t))-F(t,x,m_2(t))}_{L^\infty([0,s]\times\Omega)}\le C\sup\limits_{r\in[0,s]}\mathbf{d_1}\big(m_1(r),m_2(r)\big)\,,
$$
we have from the Corollary of \emph{Theorem IV.9.1} of \cite{ladyzenskaja1968linear},
\begin{equation}\label{corlsu}
	\begin{split}
		\norm{\nabla z}^2_{L^\infty([0,s]\times\Omega)}\le C\left[\sup\limits_{r\in[0,s]}\mathbf{d_1}\big(m_1(r),m_2(r)\big)\right]^2\,.
	\end{split}
\end{equation}
Since in case of \eqref{dirichlet} and \eqref{neumann} the domain $\Omega$ is bounded and the test functions for $\mathbf{d_1}$ are bounded, we have
\begin{multline*}
\left[\sup_{r\in[0,s]}\mathbf{d_1}\big(m_1(r),m_2(r)\big)\right]^2\le\left[\sup\limits_{r\in[0,s]}\sup\limits_{\substack{Lip(\phi)\le 1\\\norminf{\phi}\le C}}\into\phi(x)\,\mu(r)\,dx\right]^2\\
\le C\sup_{r\in[0,s]}\into\mu(r)^2\,dx\,.
     \end{multline*}
Coming back to \eqref{est_mu}, we have
$$
\sup\limits_{s\in[0,t]}\into\mu(s)^2\,dx+\inti|\nabla\mu|^2\,dx ds\le C\int_0^t\sup_{r\in[0,s]}\into\mu(r)^2\,dx\,ds\,,
$$
which implies, thanks to Gronwall's Lemma, $\mu\equiv0$, and so $z\equiv0$.

For the Cauchy problem \eqref{cauchy}, we consider two processes $(X_t^1)_t$ and $(X_t^2)_t$ with $\mathscr L(X_t^1)=m_1(t)$, $\mathscr L(X_t^2)=m_2(t)$ and 
 $X_0^1=X_0^2$. Arguing as in \eqref{eq:pregronwall}, we have
$$
\EE\left[|X_t^1-X_t^2|^2\right]\le C\int_0^t\int_{\R^n}|\nabla(u_1-u_2)|^2\,m_1(s,x)\,dxds+C\,\EE\left[\int_0^t|X_s^1-X_s^2|^2\,ds\right]\,.
$$
Applying Gronwall's Lemma, we have
\begin{multline*}
\sup\limits_{s\in[0,t]}\mathbf{d_1}\big(m_1(s),m_2(s)\big)^2\le \EE[|X_t^1-X_t^2|^2]\le C\int_0^t\int_{\R^n}|\nabla z|^2\,m_1\,dxds\\
\le C\int_0^t \norm{\nabla z}^2_{L^\infty([0,s]\times\Omega)}\, ds\,.
\end{multline*}
Using once again \eqref{corlsu}, we get
$$
\sup\limits_{s\in[0,t]}\mathbf{d_1}\big(m_1(s),m_2(s)\big)^2\le C\int_0^t\sup\limits_{r\in[0,s]}\mathbf{d_1}\big(m_1(r),m_2(r)\big)^2\,ds\,,
$$
which implies $\mu\equiv0$, and so $z\equiv0$. This concludes the proof.

\end{proof}

\section{Two Forward-Forward models}\label{s3}

In this section we introduce two models which have the structure of a FF-MFG. The framework is appropriate when we want to model a continuity of agents which are dynamically retrieving information from the past environmental scenarios. We focus on an opinion formation model set in a 1D space and a 2D voting model which share the same formal structure.

\subsection{A 1D opinion formation model}

In recent years, the process of opinion formation, along with other sociological and economic phenomena, has garnered considerable attention from physicists and mathematicians. It has been observed that certain tools employed in statistical mechanics can be effectively applied to model these phenomena. Notable works in this area include \cite{galam2005heterogeneous, sznajd2000opinion, deffuant2002can} and others. Within this context, two active fields have swiftly emerged, commonly referred to as sociophysics or, within the mathematical community, behavioral social systems. These fields are dedicated to describing these phenomena from the perspective of physicists. For an overview and up-to-date references, we recommend the monographs \cite{chakrabarti2007econophysics, galam2012sociophysics, sen2014sociophysics, BBP17}. Additionally, for works with a more mathematical focus, consider \cite{BCM14}, where the authors discuss the use of partial differential equations in modeling for the social sciences, and \cite{FPTT17}, where the authors employ Fokker-Planck equations to model socio-economic phenomena.

\medskip

In the work \cite{stella2013opinion}, the authors propose an opinion dynamic model that describes the evolution of an opinion in a population consisting of two groups: pliant and stubborn agents, where the latter act as leaders. The model adopts the structure of a multi-population backward-forward Mean Field Game.
Drawing inspiration from \cite{stella2013opinion}, we adopt standard notation and concepts, considering a population of homogeneous agents. Each agent is characterized by the opinion $X(t)\in \Omega= \mathbb{R}$ at time $t\in [0,T]$, where $[0,T]$ represents the time horizon window. We also consider the presence of some advertised opinions, as clarified further below. Notably, in the case of a Forward-Forward Mean Field Game (FF-MFG), there is no modeling or theoretical necessity to bound $T$, unless for the numerical tools that are required for approximating the solution.

\medskip

We assume that the opinion $X(t)$ is a solution of \eqref{dynamicsX}, where the dynamics are driven by the control variable $\alpha$. The presence of the stochastic term in \eqref{dynamicsX} is motivated by the effect of uncertainty in the opinion evolution. We consider a probability density function $m :\Omega\times[0,T]\rightarrow \mathbb{R}$ representing the percentage of agents in state $x$ at time $t$, satisfying 
$$ \int_\Omega m(t,x)dx=1, \quad \text{for any } t\in[0,T].$$

The fundamental principle underlying our model is that each agent adjusts his opinion based on the average opinion of their surroundings. Specifically, an agent is more sensitive to and attracted by opinions with high density that are close to their own, while they give less consideration to opinions that are too distant. This reflects a typical crowd-seeking behavior, previously studied in its non-local variation in \cite{stella2013opinion} and justified by a kinetic framework in \cite{toscani2006kinetic}. We consider a fixed probability distribution function $g(x)$ (for example, $g$ could be the density of a normally distributed random variable). For any fixed $t\in[0,T]$, we define the local average opinion as 
$$ \bar m(t,x):= \left(y m(t,\cdot)\ast g(\cdot)\right)(t,x)=C(t,x)\int_{\Omega}y m(t,y)g(y-x)dy, $$
where $C(t,x)$ is a normalization variable, i.e., $C(t,x) = 1/\int_{\Omega} m(t,y)g(y-x)dy$.

We model the running cost terms in \eqref{representacionestocastica} to penalize the distance from the average opinion of the surroundings, determined by the shape of the distribution $g$. Thus, 
\begin{equation}\label{Fopi}
F(x,m)=a_1\left|\bar m(t,x)-x\right|. 
\end{equation}

Simultaneously, we assume the existence of a collection of points $\{\bar{x}_i\}_{i=1,...,N}$ that represent special opinions where advertising is spent. These points serve as special points of attraction, and their strength depends on the parameters $k_i$. Therefore, 
\begin{equation}\label{ellopi}
\ell(x,\alpha)=\frac{1}{2}\left[a_2\alpha^2+a_3\min_{i=1,...,N}k_i\left(\bar x_i-x\right)^2\right].
\end{equation}

It's worth noting that we also penalize the control, implying that the agents are stubborn and prefer to change their opinion as little as possible. The parameters $a_1$, $a_2$, $a_3 \geq 0$ balance the influence of these various principles guiding the evolution of the agents' choices.

We consider an initial distribution of opinions $m_0(x)$ and an initial cost $u_0(x)$. It is possible to assume either that $u_0(x)\equiv 0$ if the agents start the game without any prior information about the system or 
$$ u_0(x) = \min_{i=1,...,N}k_i\left(\left(\bar x_i-x\right)^2\right)$$
if we assume that there have already been some advertisements prior to the start of the game, introducing the presence of some polarizing opinions.

The final FF-MFG model that we obtain is, for $(t,x)\in [0,T]\times \R$
\begin{equation}
\label{ffmfg-opi}
\left\{
\begin{array}{ll} 
u_t(t,x)+\max\limits_{\alpha\in\R}\left\{\alpha u_x(t,x)-\ell(x,\alpha))\right\}=F(x,m)+(A(x)\, u_{x}(t,x))_x, \\
m_t(t,x)-\left( m(t,x)H_p(x,u_x(t,x))\right)_x=(A(x)\, m_{x}(t,x))_x, \\
m(0,x)= m_0(x),\quad  u(0,x)=u_0(x)
\end{array}\right.
\end{equation}

We observe that using the standard Legendre-Fenchel transform, the control depending part of the hamiltonian is
$$ H(x,p)=\max_{\alpha\in\R}\left\{\alpha p-a_2\alpha^2/2)\right\}=\frac{1}{a_2} \,\frac{p^2}{2}.$$

\begin{lemma}
Let $A>0$ and $(m_0,u_0)$ satisfying H3 in the setting \eqref{cauchy}. Then all the hypotheses H0-H3 are verified and therefore the system \eqref{ffmfg-opi} has a unique solution.
\end{lemma}
\begin{proof}
The hypothesis H1 is obviously verified. For the hypothesis H2, observe that for any $t\in[0,T]$
\begin{multline}
\left|F(x_1,m_1)-F(x_2,m_2)\right| \leq
a_1\left| \left|\bar  m_1(t,x_1)-x_1\right|-\left|\bar  m_2(t,x_2)-x_2\right| \right| \leq \\
a_1\left( |x_1-x_2|+\max{|C_1(t,\cdot)|,|C_2(t,\cdot)|}\left|\int_\Omega y( m_1(t,y)g(y-x_1)- m_2(t,y))g(y-x_2)dy\right|\right)\\
\leq 
a_1\left(|x_1-x_2|+\max{|C_1(t,\cdot)|,|C_2(t,\cdot)|}\max(\|g\|_\infty,L_g)\du( m_1, m_2)\right)\\
\leq K(|x_1-x_2|+\du( m_1, m_2))
\end{multline}
for a $K$ big enough. Therefore we can apply Theorem \ref{T1} and conclude the proof.
\end{proof}
It is noteworthy that the model bears closer resemblance to \cite{toscani2006kinetic} and other kinetic-based frameworks than to \cite{stella2013opinion}. This is due to our assumption that agents lack access to any information about future states or changes in the dynamics of the system. Nevertheless, the model has a steady-state formulation that adopts the typical stationary MFG form. Hence, we anticipate that after a certain time, the solutions $(u, m)$ will converge to $(\bar u, \bar  m)$, which are solutions of the system:

\begin{equation}
\label{ffmfg-opi-steady}
\begin{cases}
\rho + \max\limits_{\alpha\in\R}\left\{\alpha \bar u_x(x) - \ell(x,\alpha)\right\} = F(x,\bar  m) + (A(x)\, u_{x}(x))_x, \\
 (A(x) \,m_{x}(x) )_x+ \left( \bar u_x(x) \bar  m(x)\right)_x = 0,
\end{cases}
\end{equation}

where $\rho\in\R$ is a constant of compatibility between the two systems to be defined. Proving the long-time convergence of this class of problems remains an open issue (cf. \cite{achdou2017mean}).

\subsection{FF-Opinion formation: Tests}

In this section, we highlight some of the features of the proposed model. The primary objective of the following tests is to emphasize the distinctions from the corresponding forward-backward system and to illustrate the qualitative characteristics of the model. The simulations presented are obtained using semi-Lagrangian schemes (refer to \ref{app:num.tools} and the cited references for details).

\paragraph{Test 1. A Forward-Backward/Forward-Forward comparison. }
The objective of our first test is to compare our opinion formation model with its corresponding forward-backward counterpart. It is essential to note that in the latter case, agents forecast all future states of the system and optimize their policy based on those predictions. In contrast, in our case, each agent determines their policy based on past states of the system. The forward-backward coupling introduces additional computational challenges. Due to the nature of its evolution, we cannot employ an explicit numerical scheme for both equations simultaneously. To address this, we adopt a solution strategy similar to \cite{carlini2016DGA}, utilizing a fixed-point argument until convergence to the equilibrium of the discrete system.

\medskip 

We consider the opinion domain $\Omega=\mathbb{R}$, and the initial values of the potential and density functions are given by
$$ u_0(x)=\min_{i=1,2}k_i\left(\left(\bar x_i-x\right)^2\right)$$
with $(\bar x_1,k_1)=(0.8, 1)$ and $(\bar x_2,k_2)=(0.2, 3)$, 
$$ m_0(x)= \begin{cases} 1 & \text{if }x\in[0,1]\\ 0 & \text{otherwise} \end{cases}.$$
In other words, initially, the opinion is uniformly distributed, but agents are aware of two polarizing opinions $\{0.2,0.8\}$ for which some prior advertisement may have been spent. Specifically, the parameter of the pole $\bar x_2=0.2$ is set to $k_2=3$, while the pole $\bar x_1=0.8$ has a lower value parameter $k_1=1$. The same parameters appear in the running cost $\ell(x,\alpha)$. The strength of this factor is tuned by the weight $a_3=2$, while the penalization of the control, i.e., the resistance to change an opinion, is set to $a_2=1$. 

The distribution $g(x)$ is set to $g(x)=\frac {1}{\sqrt {0.2\pi } }e^{-{\frac {x^{2}}{0.2} } }$, modeling the influence on every agent of their surroundings. The weight of this tendency to conform to the local mean is tuned by $a_1=1$.

\begin{figure}[h]
\begin{center}
\begin{tabular}{c}
Forward-Forward Model\\
\includegraphics[width=.9\textwidth]{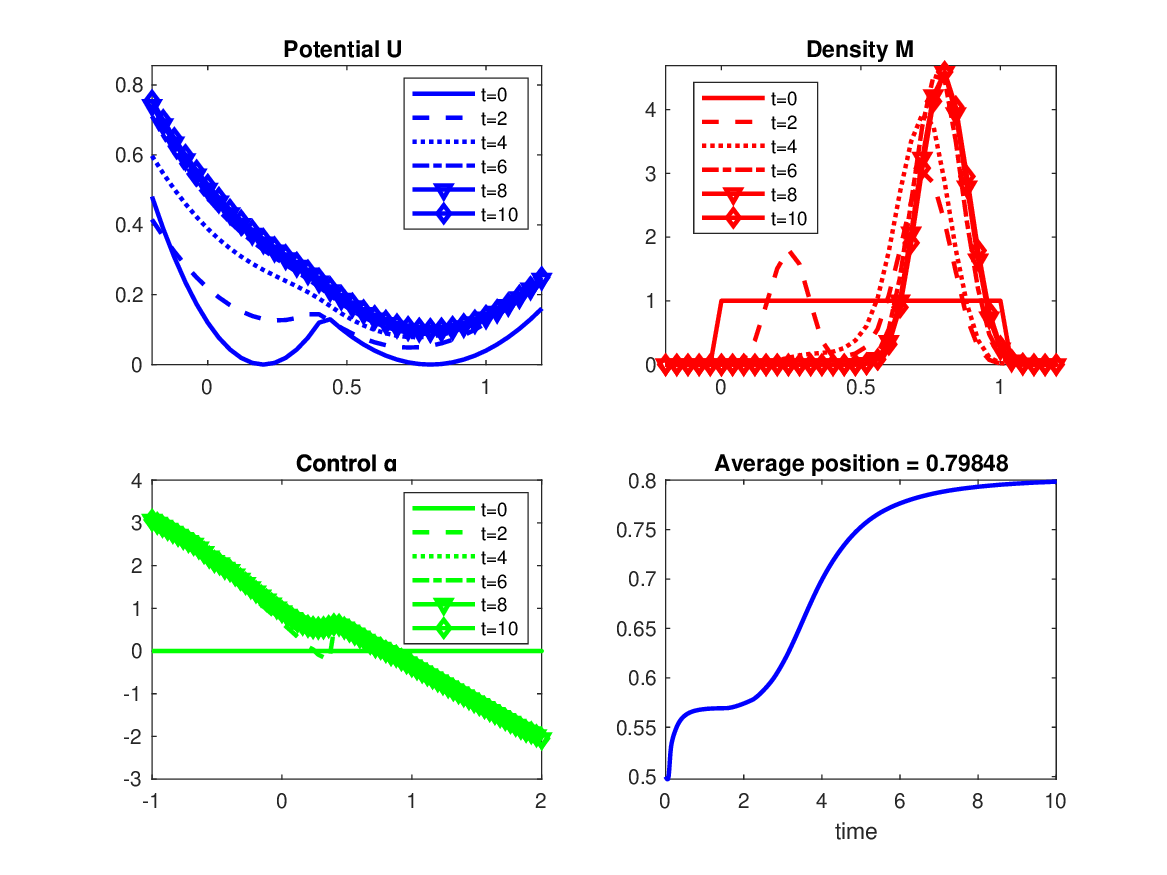}
\end{tabular}
\end{center}
\caption{Test 1. Forward-Forward model: evolution of the potential function $u$ and density distribution $ m$ (top), control map $-\nabla u/a_2$ and mean opinion $\int_{\Omega} xm(t,x)dx$. }\label{3}
\end{figure}

Finally, the diffusion is set to $A(x)\equiv 0.01$, and the discount parameter is $\lambda=0$. 

\medskip 
In Figure \ref{3}, we present the results of this initial test implemented with discretization parameters $\Delta x=0.04$ and $\Delta t=0.02$, covering approximately $\Omega\approx [-4,4]$ with Neumann boundary conditions.
We observe that the density distribution of the agents, initially constant in $[0,1]$ and null elsewhere, splits into two clusters under the influence of the initial choice of $u_0$. During the initial phase, the right-hand-side cluster progressively acquires a greater attraction capacity and tends to absorb the second cluster. Around $t=10$, this process concludes, and only one cluster is present, although it is still not centered on the advertised pole $\bar x_2$. This occurs during the final phase of the system's evolution, where the cluster slowly moves toward $\bar x_2$, becoming the center of its distribution.

Notably, there is no formation of Dirac's deltas due to the presence of a strictly positive diffusion parameter $A$. Figure \ref{3} also displays the evolution of the control map determined by the potential function $u(t,x)$ and the mean opinion of the agents, i.e., $\int_{\Omega} x m(t,x)dx$, since $\int_{\Omega} m(t,x)dx=1$. We observe the crowd reaching consensus around the pole $\bar x_2$.

\begin{figure}[h]
\begin{center}
\begin{tabular}{c}
Forward-Backward Model\\
\includegraphics[width=.9\textwidth]{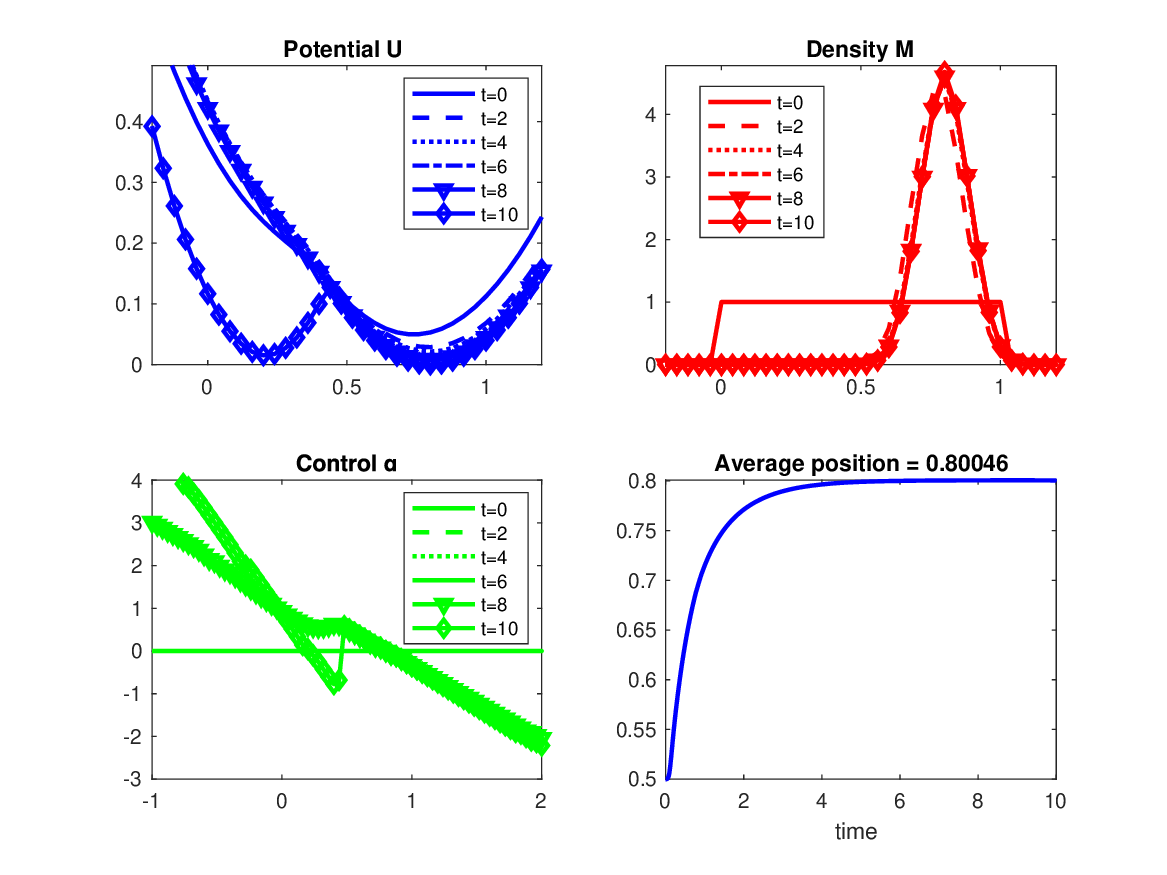}
\end{tabular}
\end{center}
\caption{Test 1. Forward-Backward model: evolution of the potential function $u$ and density distribution $ m$ (top), control map $-\nabla u/a_2$ and mean opinion $\int_{\Omega} xm(t,x)dx$.  }\label{4}
\end{figure}

\begin{figure}[ht]
\begin{center}
\begin{tabular}{cc}
\hspace{-1.2cm}
\includegraphics[width=.6 \textwidth, height=.45 \textwidth]{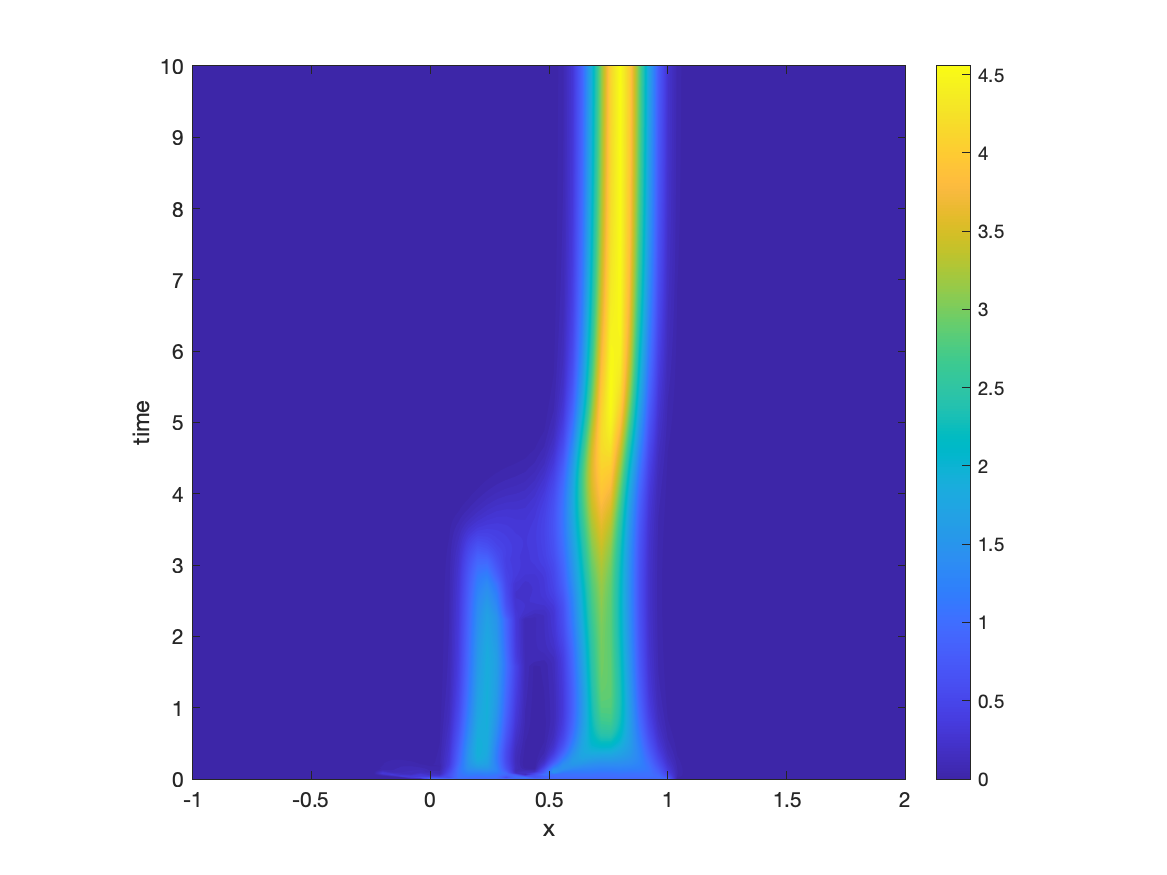}& \hspace{-1cm}
\includegraphics[width=.6\textwidth, height=.45 \textwidth]{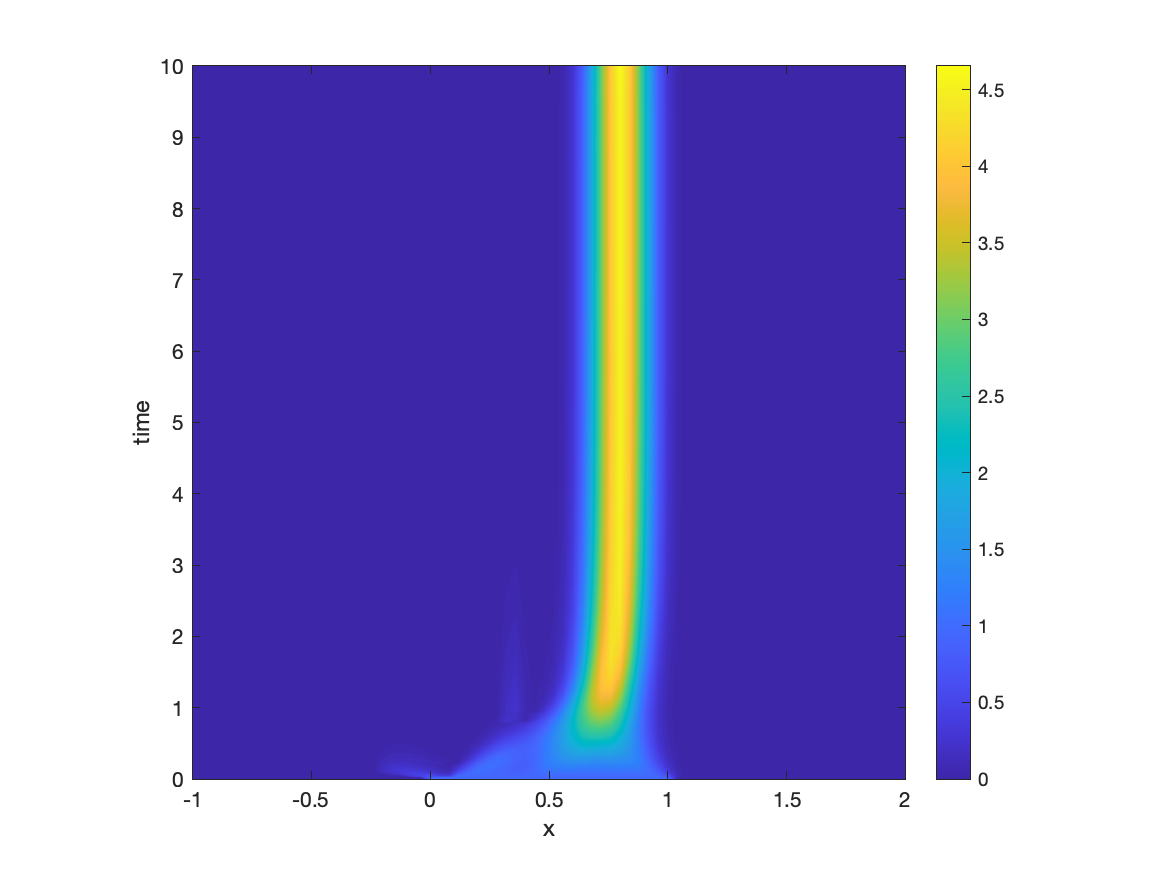} \\
Forward-Forward Model & Forward-Backward Model\\\
\end{tabular}
\end{center}
\caption{Test 1. Evolution of the density of opinion in the FF and FB model.}\label{5}
\end{figure}

\medskip
Let us compare this with the forward-backward counterpart model. We maintain all the parameters of the system unchanged, only inverting the evolution of the potential equation. The results are presented in Figure \ref{4}. We observe that the evolution of the system is markedly different. Firstly, there is no formation of two competing clusters since the agents are aware of all the future states of the system; therefore, only one attraction point emerges. Also in this case, the 'winner' of the competition is the pole $\bar x_2$, and the average opinion of the agents monotonically aligns with this final state. A more evident comparison is also reported in Figure \ref{5}, where the evolution of the density of opinions is shown side by side.

\clearpage

\paragraph{Test 2. The effect of the discount variable $\lambda$. } 

In our second test, we aim to show the effect of the discount variable $\lambda$ on the model. While it makes sense to limit the influence of the past in the optimization process (clearly, the paradigm of a player assuming that past scenarios will repeat in the future is potentially flawed), adopting a larger $\lambda$ tends to trivialize the model, transforming it into a simple relaxation around the average. This effect is illustrated in Figure \ref{6l}, where, for the same parameters as Test 1, the potential $u$ and the density distribution $m$ are shown for various choices of $\lambda$. Notably, for $\lambda\geq 7$, both $u$ and $m$ quickly move, in the early instants of the process, to a stable configuration around the poles. This dynamic, less interesting and varied than for lower values of $\lambda$, led us to restrict the remaining tests to the case $\lambda=0$. For intermediate values of lambda (e.g., $\lambda\in(0,7)$), we expect to observe essentially the same type of behavior as for $\lambda=0$, albeit mitigated and relatively smoothed down around the attraction points of the model.

\begin{figure}[h]
\begin{center}
\begin{tabular}{c}
Forward-Forward Model with various $\lambda$\\
\includegraphics[width=.7\textwidth,height=0.91 \textwidth]{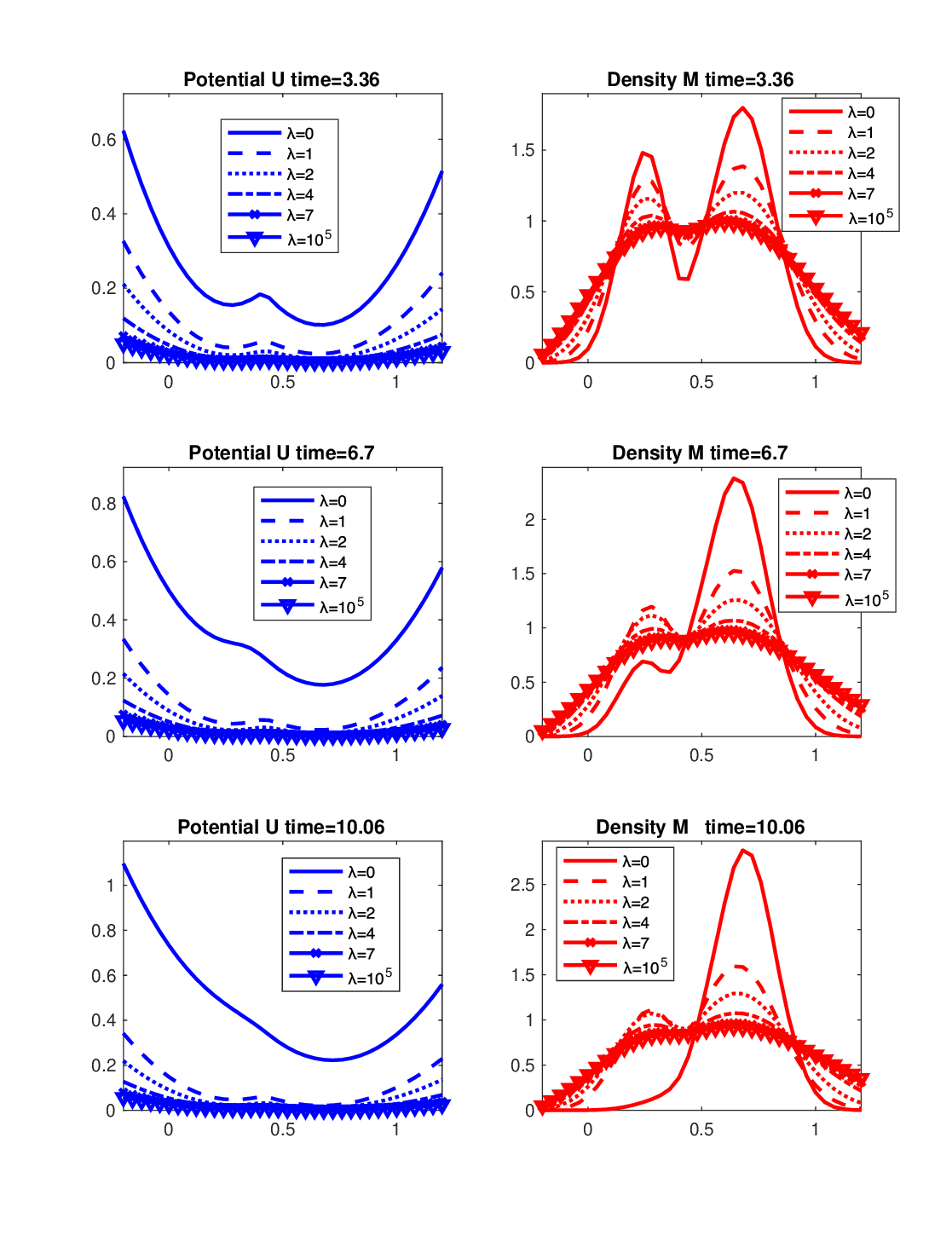}
\end{tabular}
\end{center}
\caption{Test 2. Forward-Forward model: evolution of the potential function $u$ (left) and density distribution $ m$ (top) with various values of the discount coefficient $\lambda$.  }\label{6l}
\end{figure}

\clearpage

\paragraph{Test 3. Cluster formation and control. } In our third test, we aim to illustrate two distinct features of our opinion model. In the first one, we demonstrate how our model can generate clusters over time. To achieve this, we consider the same setting as the previous test, with the only exception of the parameters $a_1=1$, $a_2=4$, and $a_3=1$. In practice, we penalize the possibility to change opinions, fostering the 'stubbornness' of our density of agents. The two attraction poles are again $0.2$ and $0.8$ with equal attraction capacities, i.e., $(\bar x_1,k_1)=(0.8, 1)$ and $(\bar x_2,k_2)=(0.2, 1)$.

\begin{figure}[h]
\begin{center}
\begin{tabular}{c}
Forward-Forward Model\\
\includegraphics[width=.9\textwidth]{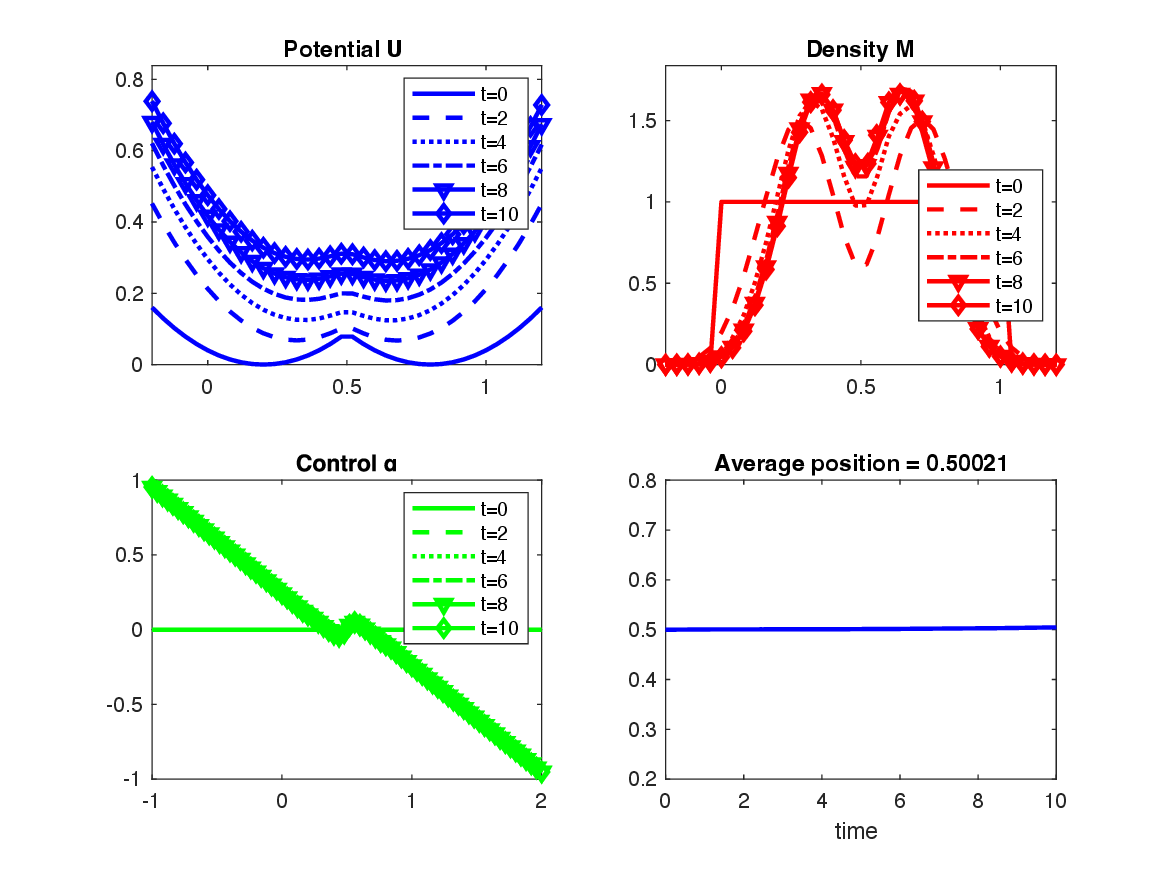}
\end{tabular}
\end{center}
\caption{Test 3. Forward-Forward model: evolution of the potential function $u$ and density distribution $ m$ (top), control map $-\nabla u/a_2$ and mean opinion $\int_{\Omega} xm(t,x)dx$. }\label{6}
\end{figure}

The test results, depicted in Figure \ref{6}, reveal a symmetric pattern around two attraction points. Initially, the agents are uniformly distributed in density, but they quickly aggregate around the two advertised poles. The problem's inherent symmetry ensures that the average opinion of all agents remains stable, with minor numerical oscillations, at the midpoint between the two clusters, denoted as $x=0.5$. Notably, in this scenario, the two clusters are interconnected, indicating that the intersection of their support is not empty. By varying the distance or making a different choice for the function $g$ (as demonstrated in subsequent tests), it's possible to achieve two or more separated clusters.

\medskip 

We delve into the question of the best advertising strategy in a relatively stable situation like the latter. Particularly, considering a limited advertising budget, is it more advantageous to be the first one to advertise at one pole, or is it better to have the ability to advertise for a longer duration? To address this question, we conduct the same test as before, setting
\begin{eqnarray*}
&\text{for } t\in[0,0.5], \quad (\bar x_1,k_1)=(0.8, 3), \quad (\bar x_2,k_2)=(0.2, 1); \\
&\text{then }t\in(0.5,10],\quad  (\bar x_1,k_1)=(0.8, 1), \quad (\bar x_2,k_2)=(0.2, 3).
\end{eqnarray*}
All the other parameters of the test are left as in the previous case.

In this scenario, the agent interested in advertising position $0.8$ acts \emph{before} and \emph{for a shorter duration}, while its competitor has a much longer advertising span but is less responsive at the beginning of the simulation. Notably, the advertisement parameters $k_1$ and $k_2$ switch their roles at $t=0.5$, representing $5\%$ of the total evolution of the system.

\begin{figure}[!h]
\begin{center}
\begin{tabular}{c}
Forward-Backward Model\\
\includegraphics[width=.9\textwidth]{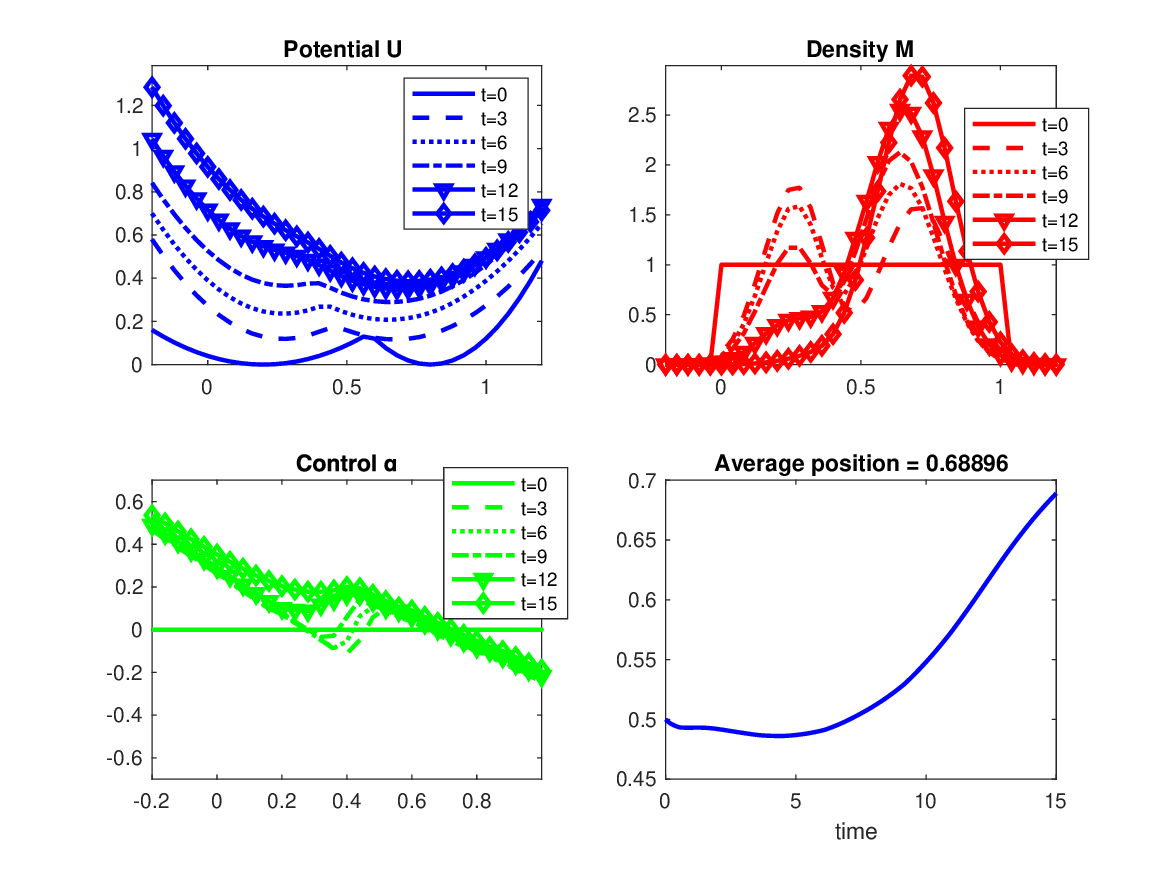}
\end{tabular}
\end{center}
\caption{Test 3. Controlled FF model: evolution of the potential function $u$ and density distribution $ m$ (top), control map $-\nabla u/a_2$ and mean opinion $\int_{\Omega} x m(t,x)dx$.}\label{7}
\end{figure}

\begin{figure}[!h]
\begin{center}
\begin{tabular}{cc}
\hspace{-1cm}
\includegraphics[width=.6 \textwidth, height=.45 \textwidth]{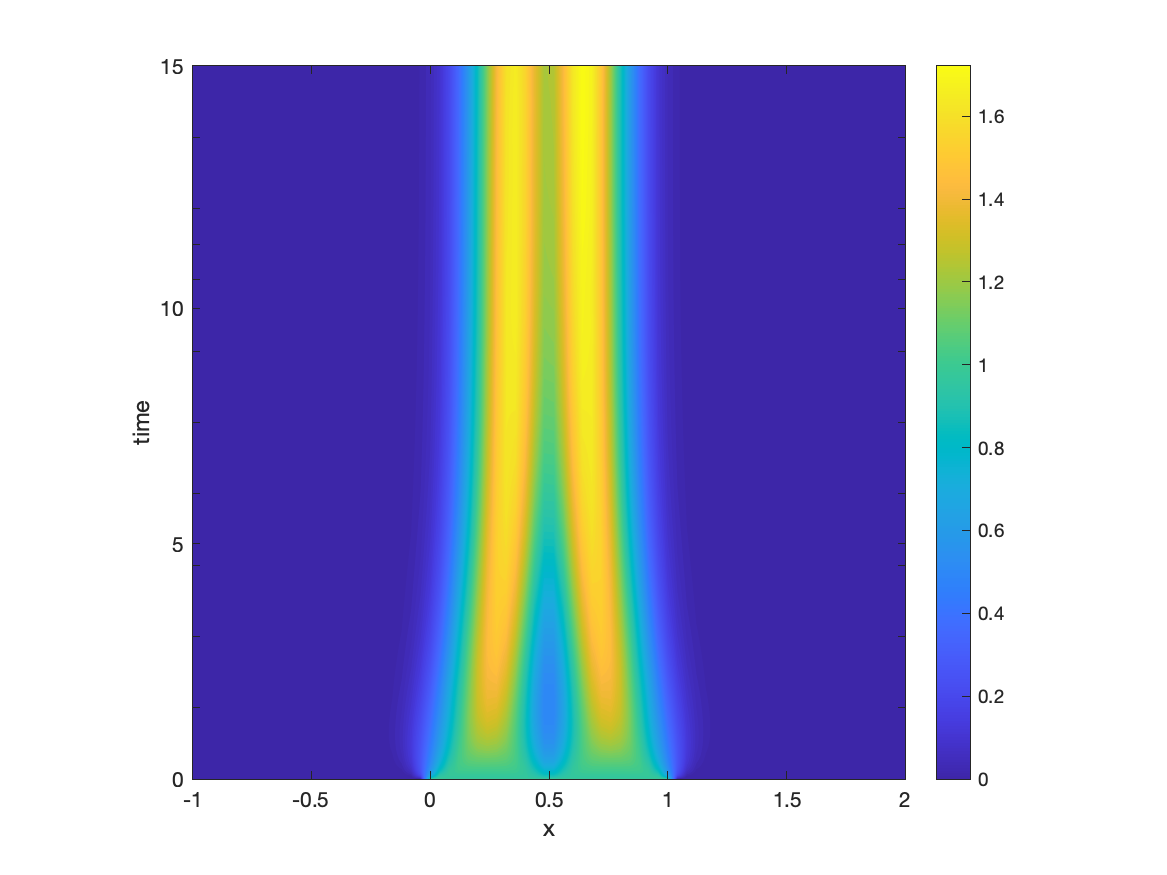}& \hspace{-1cm}
\includegraphics[width=.6\textwidth,height=.45 \textwidth]{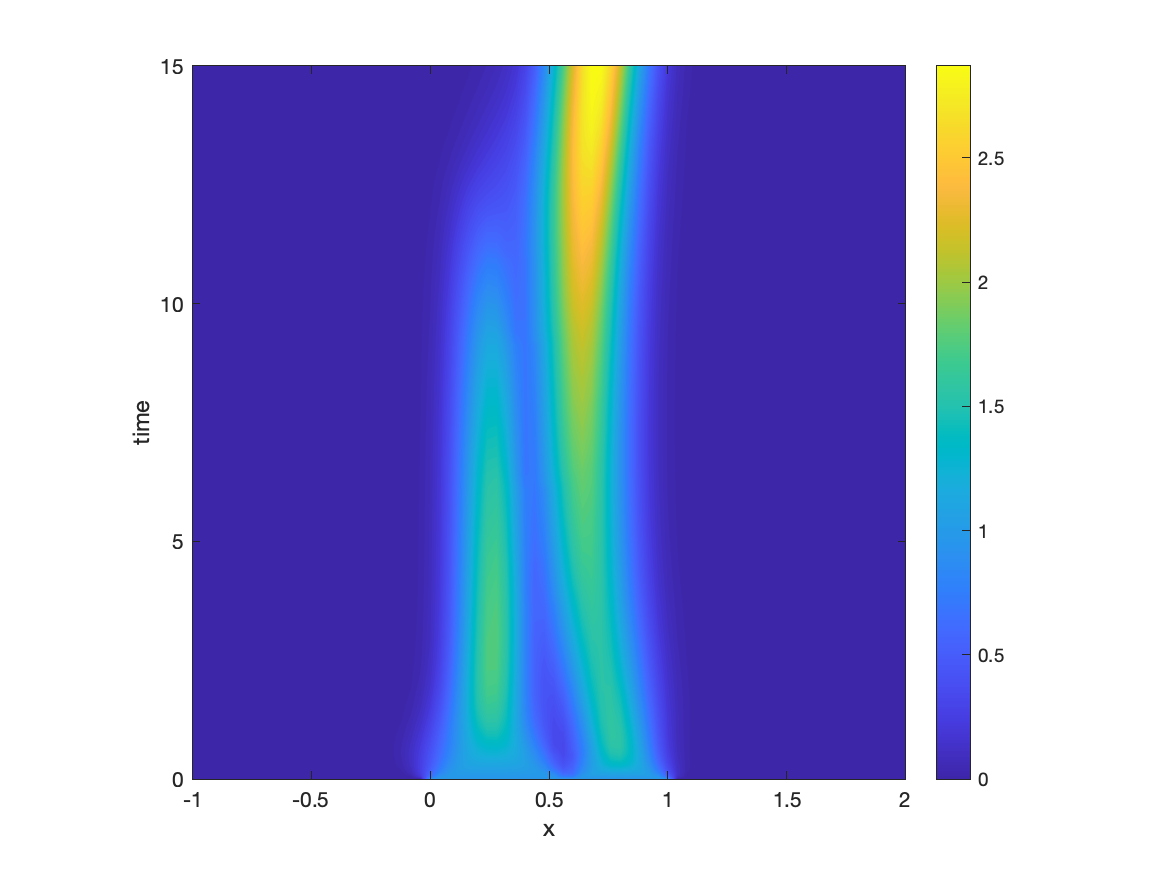} \\
Forward-Forward Model & Controlled FF Model \\\
\end{tabular}
\end{center}
\caption{Test 3. Evolution of the density of opinion in the uncontrolled and controlled model}\label{8}
\end{figure}

The results of this simulation are presented in Figures \ref{7} and \ref{8}. We observe that in the long run, the position advertised with more anticipation, $0.8$, becomes predominant, even though at $t=0.5$, when its advertising advantage ends, it is actually in a minority position if we consider the average opinion of the agents. Therefore, we can assert, at least experimentally, that in our model, advertisement in the initial evolution moments of the system is much more relevant, essentially when the first split into clusters occurs. These initial moments lead to non-trivial long-term effects of polarization, as observed in the literature \cite{lorenz2007continuous}.

\subsection{A dynamic 2D vote model}

In this section, we extend the previously introduced opinion formation model to a 2D political spectrum, often utilized to provide a more detailed ideological categorization of voters. An illustrative example of such 2D representations is the Nolan Chart \cite{bryson1968political}, where the axes represent personal freedom and economic freedom.

We assume that each voter adopts a position $X(t)\in \Omega= \R^2$ as a solution of the SDE dynamics \eqref{dynamicsX}, controlled by the variable $\alpha$. Once again, the density function $ m:[0,T]\times\Omega\rightarrow \R$ represents the percentage of voters in state $x$ at time $t$, subject to the condition 
$$ \int_{\Omega}  m(t,x)dx=1, \hbox{ for any } t\in[0,T].$$ 
Adopting the same principles as the previous 1D model, we assert that every agent adjusts their vote opinion based on an average of its surroundings. Therefore, we introduce a distribution function $g(x)$ and define, for any fixed $t\in[0,T]$, the local average as 
$$ \bar  m(t,x):= \mathbb E\left[\left(y m(t,\cdot)\ast g(\cdot)\right)(t,x)\right]=C(t,x)\int_{\Omega}y m(t,y)g(y-x)dy, $$
where $C(t,x)$ is a normalization variable, i.e., $C(t,x) = 1/\int_{\Omega}  m(t,y)g(y-x)dy$.

The distribution $g(x)$ is chosen as the isotropic multivariate normal distribution  $g(x)=g(x_1,x_2)=\frac {1}{\sqrt {\mu\pi } }e^{-{\frac {x_1^{2}+x_2^2}{\mu} } }$. More complex choices for modeling the influence of the surrounding on each voter are possible.

The other terms in \eqref{representacionestocastica} are given by
\begin{equation}
F(x,m)=a_1\left|\bar m(t,x)-x\right|.
\end{equation}
This term penalizes the distance from the average opinion of the surrounding. Assuming the existence of a collection of points $\{\bar{x}_i\}_{i=1,...,N}\in (\R^2)^n$ representing "advertised political opinions" or policy-seeking behavior of various parties with strengths $k_i$,
\begin{equation}
\ell(x,\alpha)=\frac{1}{2}\left[a_2\alpha^2+a_3\min_{i=1,...,N}k_i   \|\bar x_i-x\|^2\right].
\end{equation}
Here, $a_1$, $a_2$, and $a_3$ (all greater than or equal to 0) model \emph{conformism} (to the political environment), \emph{stubbornness}, and \emph{advertisement-influenced behavior} of each voter, respectively.

As initial scenarios, we consider a uniform initial distribution of opinions $ m_0(x)$ and a constant initial cost $v_0(x)$, i.e.,
$$  m_0(x)=\frac{1}{4}\chi_{[-1,1]^2}(x), \quad  u_0(x)=0,$$
where $\chi_B(x)$ is the indicator function of the set $B$. Setting a constant diffusion parameter $A(x)\equiv A$, we obtain the following FF-MFG model for $(t,x)\in [0,T]\times \Omega$:
\begin{equation}
\label{ffmfg-opi-2D}
\left\{
\begin{array}{ll} 
u_t(t,x)+\max\limits_{\alpha\in\R^2}\left\{\alpha\cdot \nabla u(t,x)-\ell(x,\alpha))\right\}=F[ m(t,\cdot)]+A\, \Delta u(t,x), \\
m_t(t,x)-\div\left(\nabla u(t,x) m(t,x)\right)=A\,\Delta  m(t,x), \\
m(0,x)= m_0(x),\quad  u(0,x)=u_0(x)
\end{array}\right.
\end{equation}
Also in this case, the standard hypotheses on $H(x,p)$ are verified.

\subsection{FF-dynamic vote model: Tests}

We present simulations conducted using the same schemes outlined in \ref{app:num.tools}, focusing on a 2D opinion of voting scenario with 2 candidates. The primary objective of these tests is to demonstrate the effectiveness of the proposed model in simulating complex phenomena such as cluster formation or the adoption of classic strategies to enhance the electoral performance of a candidate.

\paragraph{Test 4: Stable Cluster Formation between Two Candidates}

In this 2D setting, we showcase a fundamental cluster formation where voters polarize between two candidates. It is essential to note that our model, relying on attraction poles, naturally tends towards the formation of clusters over the long run. For this test and the following one, model parameters $a_1$, $a_2$, and $a_3$ are all set to one. The constant $\mu$, representing the size of the area related to the conformism tendency of a voter, is set to 0.2, while the diffusion parameter $A$ is set to 0.01. We assume the presence of two candidates, $C1$ and $C2$, positioned at $(0.8,0.8)$ and $(-0.8,-0.8)$, respectively. Both candidates conduct a political campaign with constant strength and efficiency, setting $k_1=k_2=1$.

\begin{figure}[h]
\begin{center}
\begin{tabular}{c}
\includegraphics[width=1 \textwidth, height=.52 \textwidth]{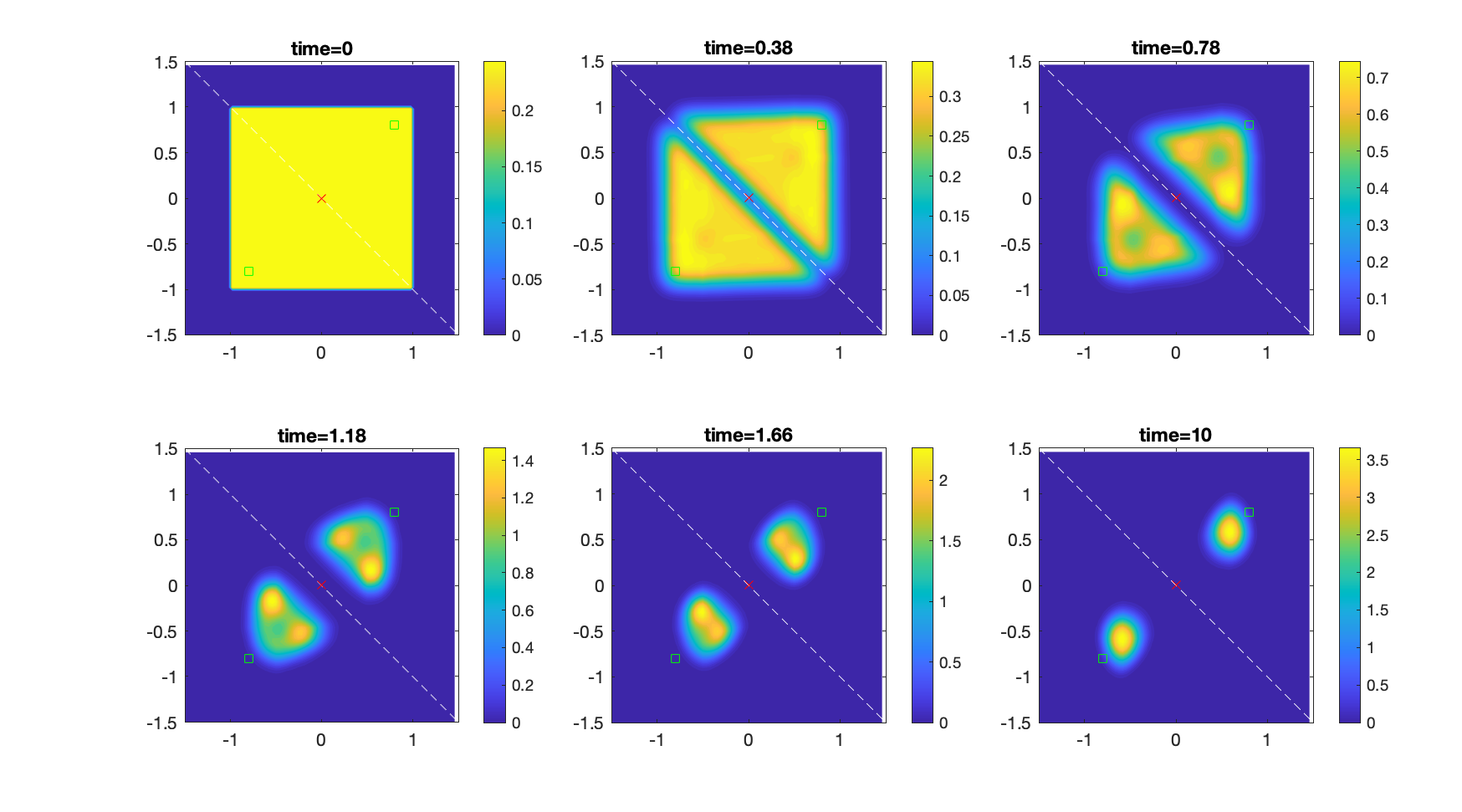}
\end{tabular}
\end{center}
\caption{Test 4. Evolution of the density of voters at $t=0, 0.38, 0.78, 1.18, 1.66, 10$. The green squares are the position of the two candidates, the red cross is the median voter position and the dotted white line is the separation between the victory region of each candidate.}\label{9}
\end{figure}

Since voters are initially evenly distributed across the entire political spectrum $[-1,1]^2$, and the potential $v_0$ is initially constant, no prior knowledge of the candidates' positions is assumed. Following the median voter theorem \cite{rowley1984relevance}, victory of one candidate over its opponent is achieved when the average of all voters' positions (i.e., the median voter) falls within the respective victory regions, which are denoted as follows:
$$ \Omega_1:= \{x=(x_1,x_2)\in 	\R^2 \;|\; x_2>-x_1\}$$
$$ \text{and } \Omega_2:= \{x=(x_1,x_2)\in 	\R^2 \;|\; x_2<-x_1\}.$$

We approximate the solution using semi-Lagrangian schemes with $\Delta x=0.04$ and $\Delta t=0.02$. The problem is solved on a grid constructed in $[-3/2,3/2]^2$ with Neumann boundary conditions. In Figure \ref{9}, the outcomes are presented. Several interesting features of the model are observed: firstly, even starting with no prior knowledge of the candidates' positions and no previous advertisements, the initial moments of system evolution ($t\in(0,1)$) significantly influence the final results. The density of voters quickly splits equally between the two candidates due to the symmetry of the data. Additionally, the conformism tendency generates multiple clusters that progressively migrate toward the advertised positions of the candidates. During this migration, multiple clusters may merge. In this scenario, we assume infinite advertisement strength from each candidate (constant in time), and as $t\gg 0$, the density distribution of voters cannot differ from one (or more) clusters at the advertised positions of the candidates. In this case, there is a perfect balance between the two candidates, and as evident from the position of the median voter (the red cross in Fig. \ref{9}), neither candidate prevails over the opponent.

\paragraph{Test 5. Overcoming the impasse via different political strategies.}

We explore the possibility of emulating typical electoral strategies in our model to overcome the impasse caused by the symmetric setting in the previous test. One potential strategy involves creating an electoral alliance to break the deadlock. How does the introduction of a third candidate $C3$, positioned at $(0.6,-0.2)$, influence the system's evolution? For this third candidate, we assume a very limited attraction strength by setting $k_3=0.2$.

\begin{figure}[!h]
\begin{center}
\begin{tabular}{c}
\includegraphics[width=1\textwidth, height=.52\textwidth]{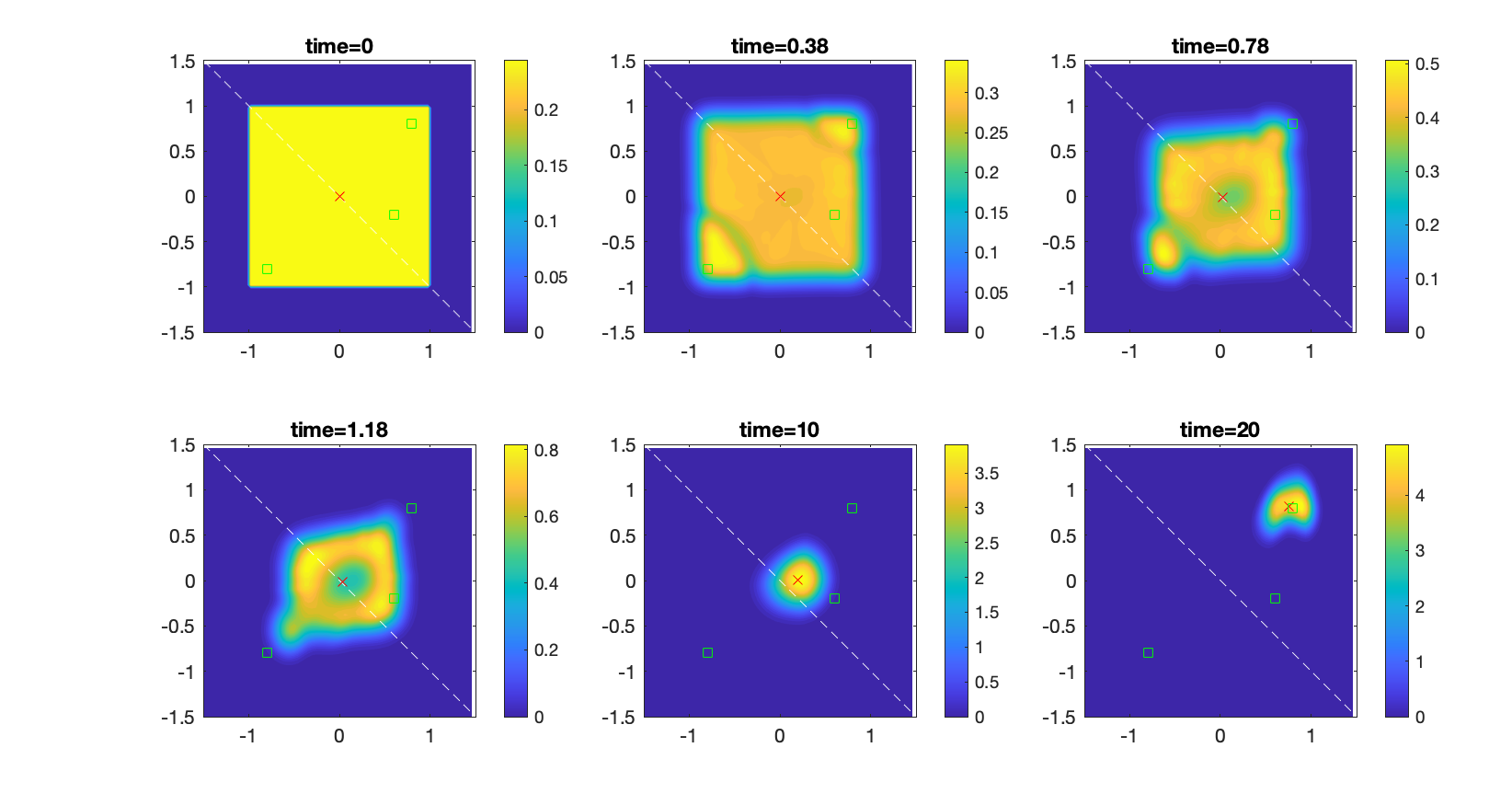}
\end{tabular}
\end{center}
\caption{Test 5.a. 'Having a moderate ally' strategy: Evolution of the voter density at $t=0, 0.38, 0.78, 1.18, 10, 20$. The green squares depict the positions of the three candidates, the red cross indicates the median voter position, and the dotted white line represents the boundary between the victory regions of each major candidate.
}\label{10}
\end{figure}

The results are presented in Figure \ref{10}. The presence of the additional candidate $C3$ plays a crucial role in the evolution of the voters. It facilitates the formation of a large central cluster, which leans slightly more towards the victory region $\Omega_1$, where both $C1$ and $C3$ are positioned. Although this central cluster initially has a limited impact on the position of the median voter, it gradually attracts voters who initially leaned towards $C2$. Subsequently, the cluster moves, passing through the position of candidate $C3$ before eventually reaching, over an extended time horizon ($t=20$), the more extreme position of candidate $C1$.

\begin{figure}[!h]
\begin{center}
\begin{tabular}{c}
\includegraphics[width=1 \textwidth, height=.52 \textwidth]{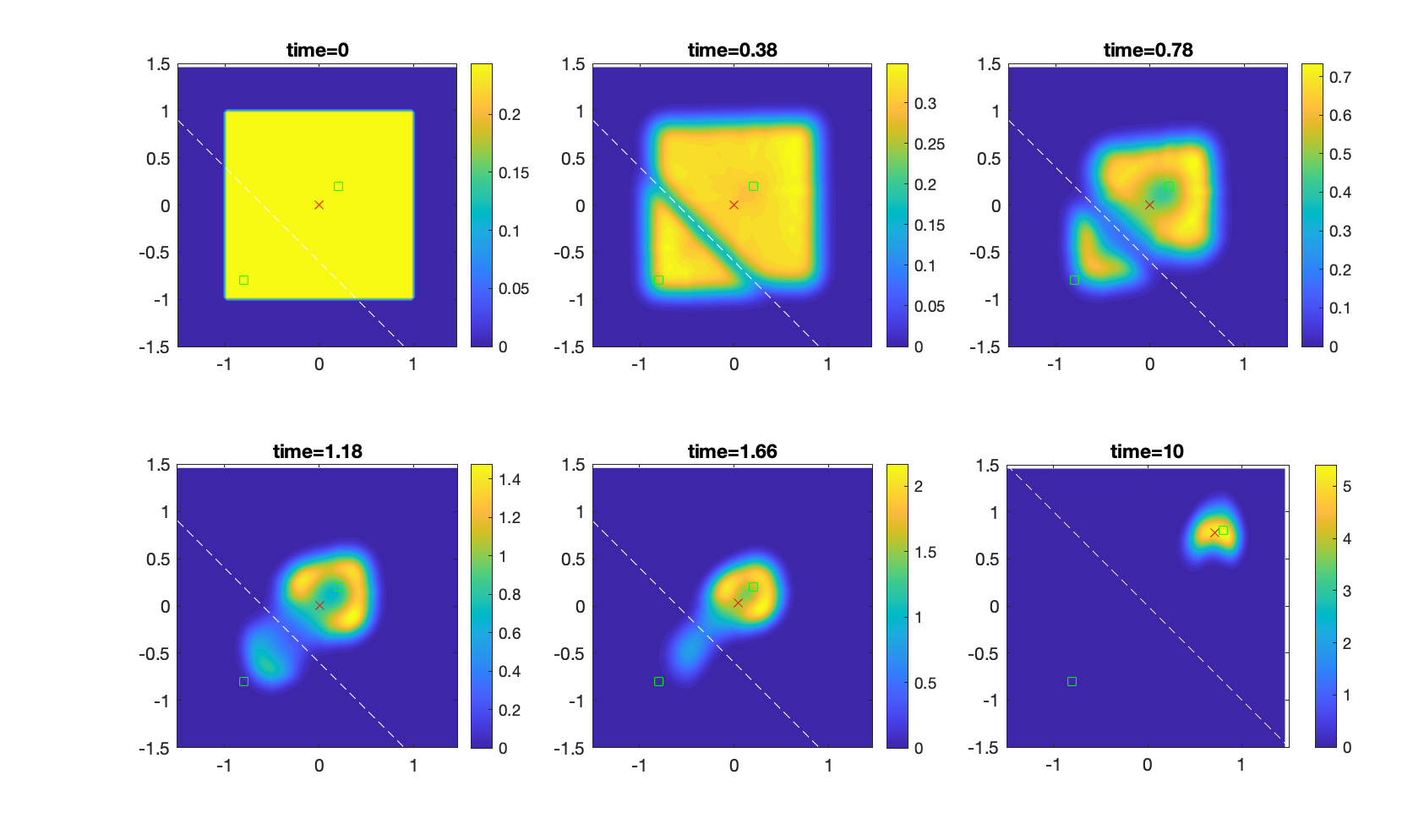}
\end{tabular}
\end{center}
\caption{Test 5.b. 'Advertising a more moderate position' strategy. Evolution of the density of voters at $t=0, 0.38, 0.78, 1.18, 1.66, 10$. The green squares are the advertised position of the two candidates (we note that one candidate changes progressively its position), the red cross is the median voter position and the dotted white line is the separation between the victory region of each candidate.}\label{11}
\end{figure}

\medskip

We explore another potential strategy: What if a candidate advertises a position that is not its final goal? In this deliberately illustrative example, we assume that the position of candidate $C1$ initially lies at the point $(0.2,0.2)$ but switches to $(0.4,0.4)$ at $t=2$, to $(0.6,0.6)$ at $t=4$, and finally to $(0.8,0.8)$ (the candidate's ultimate goal) at $t=6$. Figure \ref{11} visually demonstrates how this strategy immediately influences the position of the median voter, effectively guiding the crowd of voters toward the final goal of $C1$. This occurs even with a minor cluster formation before reaching the candidate $C1$'s final position. The tendency of politicians to gravitate toward the position occupied by the median voter, or more generally, toward the position favored by the electoral system, is sometimes referred to as Hotelling's law \cite{b133caf7-e88c-30dd-bb26-90feeddec434}. This test can also be viewed as a continuous version of the McKelvey–Schofield theorem \cite{MCKELVEY1976472}, where the authors observe that any position on a multidimensional political spectrum can be reached by a sequence of majority votes.

\section{Conclusions}\label{s4}
In this work, we discuss the fundamental features of a general FF-MFG, illustrating how such a framework can be utilized to formulate mathematical models in scenarios where a large population of rational individuals is influenced by their past position and control. This framework is particularly suitable for social mathematical models, allowing the inclusion of non-local interaction terms, diffusion effects, complex strategic behaviors, agent stubbornness, and more.

\appendix
\section{A semi-la\-gran\-gian scheme for  FF-MFG}
\label{app:num.tools}
In this appendix we provide more details about the numerical schemes that we used for approximating the FF-MFG system. To simplify a bit the presentation we consider the case where $\Omega\subset\R^2$, $A(x)$ is a constant, i.e. $A(x)\equiv \varepsilon$ and the Hamiltonian has the form \eqref{hamiltonian}. The general case is straightforward and does not change the treatment of the non-linear part of the equations. 

Our numerical approach relies on the relation between the HJ framework and the corresponding adjoint FP equation. Given a semi-discrete (discrete in space) numerical scheme for the HJ part of \eqref{MFG}, the same scheme can be used to construct a consistent approximation for the FP equation.

Before proceeding, we define additional notation. Let $\Omega_{\Delta x}$ be an uniform grid on $\Omega$ with constant discretization parameter $\Delta x>0$. Let $x_{i,j}$ denote a generic point in $\Omega_{\Delta x}$.  The space of  grid functions defined on  $\Omega_{\Delta x}$ is denoted by $\mathcal{G}(\Omega_{\Delta x})$, and the functions $U$ and $M$ in $\mathcal{G}(\Omega_{\Delta x})$ (approximations of respectively $u$ and $m$)  are called $U_{i,j}$ and $M_{i,j}$, when evaluated at $x_{i,j}$.

We utilize a semi-discrete numerical scheme $N(x,p):\Omega_{\Delta x}\times \RR^d\rightarrow \mathbb R$ monotone and consistent to approximate the operator $H(x,p)$, such that $U$ is the solution of the ODE  
\begin{equation}
U_t=N(x,\mathcal D U),
\end{equation}
where $\mathcal D U$ is a discretization of the gradient operator on $U$.
Thanks to the adjoint structure, we modify this scheme to approximate the FP part. The discrete approximation, $M$, is the solution of the following system of ODE 
$$ M_t = K(x,\mathcal D U,M),$$
where
\begin{equation}\label{fp3}
K(x,\mathcal D U,M):=(D_p N(x,\mathcal D U))^T M-\varepsilon \Delta_d M.
\end{equation}
Here, the nonlinear part of the operator corresponds to
the discrete operator $D_p N(x,\mathcal D U)$;  $\Delta_d M$ is a discretization of the Laplacian. We note that this operator depends on the monotone approximation scheme used to discretize the HJ equation, and can be computed numerically or using a symbolic differentiation operator. We stress that the features of positivity and mass conservation are still holding also at the discrete level. 
\paragraph{Semi-Lagrangian scheme.}
To describe a semi-Lagrangian scheme appropriate to approximate \eqref{MFG}, we introduce the operator 

\begin{equation}
\mathcal D^{\alpha}U_{i,j}:=\max_{\alpha\in \R^2}\frac{\mathcal I[U](x_{i,j},\alpha)-U(x_{i,j})}{h}-\ell(x_{i,j},\alpha),
\end{equation}
where $h$ a parameter of the same order of $\sqrt{\Delta x}$ and 
\begin{equation} \label{char}
\mathcal I [U](x_{i,j},\alpha)=\frac{1}{2}\sum_{i=1}^2 \left( \mathbb I[U](x_{i,j}+\alpha h +e_i\sqrt{2\varepsilon h})\right.\\
\left.+\mathbb I[U](x_{i,j}+\alpha h -e_i\sqrt{2\varepsilon h})\right).
\end{equation}
Here, $\mathbb I[u](x)$ is an interpolation operator on the matrix $U$, and $e_i$ is the $i$ unitary vector of an orthonormal basis of the space. The discrete operator is then simply
\begin{equation*} 
N(x,\mathcal D U^{\alpha}_{i,j})=F(t,x_{i,j},M)+(\mathcal D^{\alpha}U_{i,j})^\alpha.
\end{equation*}
We take the adjoint of the linearized of $N$ 
and we use it into \eqref{fp3}, 
\begin{multline*}
K(x_{i,j},(\mathcal D^\alpha U]_{i,j},M_{i,j})=\\
\frac{1}{\Delta x} \left(M_{i,j}\frac{\partial N}{\partial p_1}(x_{i,j},[\mathcal D^\alpha U]_{i,j})-M_{i-1,j}\frac{\partial N}{\partial p_1}(x_{i-1,j},[\mathcal D^\alpha U]_{i-1,j})\right.\\
+M_{i+1,j}\frac{\partial N}{\partial p_2}(x_{i+1,j},[\mathcal D^\alpha U]_{i+1,j})-M_{i,j}\frac{\partial N}{\partial p_2}(x_{i,j},[\mathcal D^\alpha U]_{i,j})\\
+M_{i,j}\frac{\partial N}{\partial p_3}(x_{i,j},[\mathcal D^\alpha U]_{i,j})-M_{i,j-1}\frac{\partial N}{\partial p_3}(x_{i,j-1},[\mathcal D^\alpha U]_{i,j-1})\\
\left.+M_{i,j+1}\frac{\partial N}{\partial p_4}(x_{i,j+1},[\mathcal D^\alpha U]_{i,j+1})-M_{i,j}\frac{\partial N}{\partial p_4}(x_{i,j},[\mathcal D^\alpha U]_{i,j})  \right).
\end{multline*}
We note that the operator $N(x,p)$ is for contruction monotone. \\
For more details cf. \cite{CS17,FF13,FGV18}.

\section*{Declarations}
\paragraph{Funding and/or Conflicts of interests/Competing interests}
AF is a member of the 'Gruppo Nazionale per il Calcolo Scientifico' INdAM and he was partially supported by PRIN project 2022 funds 2022238YY5
"Optimal control problems: analysis, approximation ". The authors have no relevant financial or non-financial interests to disclose. 
\paragraph{Author Contributions}
All authors contributed to the study conception and design. All authors read and approved the final manuscript.
\paragraph{Data Availability} No data associated in the manuscript.
\paragraph{Conflict of interest} Not applicable.
\paragraph{Compliance with Ethical Standards}
The research does not involve human participants or animals.


\begin{thebibliography}{10}

\bibitem{achdou2017mean}
Y. Achdou, M. Bardi, and M. Cirant.
\newblock Mean field games models of segregation.
\newblock {\em Math. Meth. Appl. Sci.},
{\bf 27} (2017) 75--113.

\bibitem{achdou2010mean}
Y. Achdou and I. Capuzzo-Dolcetta.
\newblock Mean field games: numerical methods.
\newblock {\em SIAM J. Numer. Anal.},  {\bf 48} (2010) 1136--1162.

\bibitem{BCD97}
M. Bardi and I. Capuzzo-Dolcetta.
\newblock {\em Optimal control and viscosity solutions of
Hamilton-Jacobi-Bellman equations}.
\newblock (Springer Science \& Business Media, 2008).

\bibitem{BBP17}
N. Bellomo, F. Brezzi, and M. Pulvirenti.
\newblock Modeling behavioral social systems. 
\newblock{\em Math. Meth. Appl. Sci.}
{\bf 27} (2017) 1--11. 


\bibitem{BCM14}
M. Burger, L. Caffarelli and P. Markowich, 
\newblock Partial differential equation models in the
socio-economic sciences, 
\newblock \emph{Philos. Trans. Roy. Soc. A} {\bf 372} (2014), Paper No. 20130406.


\bibitem{Cardialaguet10}
P.~Cardaliaguet.
\newblock Notes on {M}ean {F}ield {G}ames: from {P.-L.} {L}ions' lectures at
{C}oll\`ege de {F}rance.
\newblock (2013).

\bibitem{cardaliaguet2015weak}
P. Cardaliaguet.
\newblock Weak solutions for first order mean field games with local coupling.
\newblock In {\em Analysis and geometry in control theory and its
applications}, (2015) 111--158.

\bibitem{CS17}
E. Carlini, F.J. Silva.
\newblock  A fully discrete Semi-Lagrangian scheme for a first order mean field game problem
\newblock \emph{SIAM Journal on Numerical Analysis}, {\bf 52} (1), (2014)  45--67.


\bibitem{carlini2016DGA}
E. Carlini, A. Festa, F.~J Silva, and M.-T. Wolfram.
\newblock A semi-lagrangian scheme for a modified version of the Hughes’
model for pedestrian flow.
\newblock {\em Dyn. Games Appl.}, (2016) 1--23.

\bibitem{chakrabarti2007econophysics}
B.~K Chakrabarti, A. Chakraborti, and A. Chatterjee.
\newblock {\em Econophysics and sociophysics: trends and perspectives}.
\newblock (John Wiley \& Sons, 2007).

\bibitem{b133caf7-e88c-30dd-bb26-90feeddec434}
A. Chakraborty, P. Ghosh, and J. Roy.
\newblock Expert-captured democracies.
\newblock {\em Am. Econ. Rev.}, {\bf 110} (2020) 1713--1751.

\bibitem{crandall}
Crandall, M.G., Lions, P.-L.
\newblock Viscosity solutions of Hamilton-Jacobi equations.
\newblock {\em Transactions of the American Mathematical Society}, {\bf 277} (1983), 1--42.

\bibitem{deffuant2002can}
G. Deffuant, F. Amblard, G. Weisbuch, and T. Faure.
\newblock How can extremism prevail? a study based on the relative agreement
interaction model.
\newblock {\em J. Artif. Soc. Soc. Simul.}, {\bf 5} (2002).

\bibitem{MED}
L. Di~Persio, M. Garbelli, and M. Ricciardi.
\newblock The master equation in a bounded domain with absorption.
\newblock {\em Arxiv}, arxiv:2203.15583.

\bibitem{FF13}
M. Falcone and R. Ferretti.
\newblock {\em Semi-Lagrangian Approximation Schemes for Linear and
Hamilton—Jacobi Equations}.
\newblock (SIAM Philadelphia, 2013).

\bibitem{FGV18}
A. Festa, D.A. Gomes, R.M. Velho.
\newblock An Adjoint-Based Approach for a Class of Nonlinear Fokker-Planck Equations and Related Systems
 \newblock {\em Springer INdAM Series}, {\bf 28}, (2018)  73--92.

\bibitem{FPTT17}
G. Furioli, A. Pulvirenti, E. Terraneo and G. Toscani, 
\newblock Fokker--Planck equation in the modeling of socio-economic phenomena, 
\newblock {\em  Math. Models Methods Appl. Sci.} {\bf 27} (2017) 115--158.

\bibitem{galam2005heterogeneous}
S. Galam.
\newblock Heterogeneous beliefs, segregation, and extremism in the making of
public opinions.
\newblock {\em Phys. Rev. E .}, {\bf 71} (2005) 046123.

\bibitem{galam2012sociophysics}
S. Galam.
\newblock {\em Sociophysics: a physicist's modeling of psycho-political
phenomena}.
\newblock (Springer Science \& Business Media, 2012).

\bibitem{GomSau14}
D.~Gomes and J.~Sa{\'u}de.
\newblock Mean field games models---a brief survey.
\newblock {\em Dyn. Games Appl.}, {\bf 4} (2014) 110--154.

\bibitem{Diogo2018}
D. Gomes and M. Sedjro.
\newblock One-dimensional, forward-forward mean-field games with congestion.
\newblock {\em Discrete Cont. Dyn.-S}, {\bf 11} (2018) 901--914.

\bibitem{gomes2016one}
D.~A Gomes, L. Nurbekyan, and M. Sedjro.
\newblock One-dimensional forward--forward mean-field games.
\newblock {\em Appl. Math. Optim.}, {\bf 74} (2016) 619--642.

\bibitem{gomes2015regularity}
D.~A Gomes and E.~A Pimentel.
\newblock Regularity for mean-field games systems with initial-initial boundary
conditions: the subquadratic case.
\newblock In {\em Dynamics, Games and Science}, 291--304. Springer, 2015.

\bibitem{gueant2011mean}
O. Gu{\'e}ant, J.-M. Lasry, and P.-L. Lions.
\newblock Mean field games and applications.
\newblock In {\em Paris-Princeton lectures on mathematical finance 2010}, pages
205--266. Springer, 2011.

\bibitem{hughes}
R. Hughes.
\newblock  The flow of large crowds of pedestrians
\newblock {\em  Math. Comput. Simul.}, {\bf 53} (4) (2000) 367--370.

\bibitem{ladyzenskaja1968linear}
O-A Ladyzenskaja, V-A Solonnikov, N-N Uralceva.
\newblock Linear and quasilinear equations of parabolic type.
\newblock American Mathematical Society, 1968.

\bibitem{LasryLions07}
J.-M. Lasry and P.-L. Lions.
\newblock Mean field games.
\newblock {\em Jpn. J. Math.}, {\bf 2} (2007) 229--260.

\bibitem{sznitman}
P.-L. Lions, and A.S. Sznitman.
\newblock Stochastic differential equations with reflecting boundary conditions.
\newblock {\em Communications on Pure and Applied Mathematics}, {\bf 27} (1984) 511--537.


\bibitem{lorenz2007continuous}
J. Lorenz.
\newblock Continuous opinion dynamics under bounded confidence: A survey.
\newblock {\em Int. J. Mod. Phys. C}, {\bf 18} (2007) 1819--1838.

\bibitem{lunardi}
A.~Lunardi.
\newblock {\em Analytic Semigroups and Optimal Regularity in Parabolic
Problems}.
\newblock Modern Birkh\"{a}user Classics., 2012.

\bibitem{MCKELVEY1976472}
R.~D McKelvey.
\newblock Intransitivities in multidimensional voting models and some
implications for agenda control.
\newblock {\em J. Econ. Theory}, {\bf 12} (1976) 472--482.

\bibitem{Piccoli2}
B.~Piccoli and F.~Rossi.
\newblock On properties of the generalized Wasserstein distance.
\newblock {\em Arch. Ration. Mech. Anal.}, {\bf 222} (2016) 1339–1365.

\bibitem{Priola}
A.~Porretta and E.~Priola.
\newblock Global Lipschitz regularizing effects for linear and nonlinear
parabolic equations.
\newblock {\em J. Math. Pures Appl. },
{\bf 100} (2013) 633--686.

\bibitem{porretta2015weak}
A. Porretta.
\newblock Weak solutions to Fokker--Planck equations and mean field games.
\newblock {\em Arch. Ration. Mech. Anal.}, {\bf 216} (2015) 1--62.



\bibitem{convergenceneumann}
M. Ricciardi.
\newblock The convergence problem in mean field games with Neumann boundary conditions
\newblock {\em SIAM J. Math. Anal.} {\bf 55} (2023) 3316--3343.


\bibitem{MEN}
M. Ricciardi.
\newblock The master equation in a bounded domain with Neumann conditions.
\newblock {\em Commun. Partial. Differ. Equ.}, {\bf 47} (2022) 912--947.

\bibitem{rowley1984relevance}
C.~K. Rowley.
\newblock The relevance of the median voter theorem.
\newblock {\em J. Inst. Theor. Econ.}, {\bf 1} (1984) 104--126.

\bibitem{sen2014sociophysics}
P. Sen and B.~K. Chakrabarti.
\newblock {\em Sociophysics: an introduction}.
\newblock (Oxford University Press, 2014).



\bibitem{stella2013opinion}
L. Stella, F. Bagagiolo, D. Bauso, and G. Como.
\newblock Opinion dynamics and stubbornness through mean-field games.
\newblock In {\em Decision and Control (CDC), 2013 IEEE 52nd Annual Conference
on}, pages 2519--2524. IEEE, 2013.

\bibitem{bryson1968political}
K. Sznajd-Weron and J. Sznajd,
\newblock Who is left, who is right?, 
\newblock \emph{Phys. A: Stat. Mech.} {\bf 351} (2005) 593--604.

\bibitem{sznajd2000opinion}
K. Sznajd-Weron and J. Sznajd.
\newblock Opinion evolution in closed community.
\newblock {\em Int. J. Mod. Phys. C}, {\bf 11} (2000) 1157--1165.

\bibitem{toscani2006kinetic}
G. Toscani.
\newblock Kinetic models of opinion formation.
\newblock {\em Comm. Math. Sci.}, {\bf 4} (2006) 481--496.

\end{thebibliography}
\end{document}